\documentclass[11pt]{article}
\usepackage{amsmath}    
\usepackage{amssymb}    
\usepackage{amsthm}     
\usepackage{amscd}      
\usepackage{mathtools}  
\usepackage{dsfont}     
\usepackage{latexsym}   
\usepackage{mathrsfs}   
\usepackage{bm}         

\usepackage{graphicx}      
\usepackage{float}         
\usepackage{caption}       
\usepackage{subcaption}    
\usepackage{xcolor}        

\usepackage[sort&compress]{natbib}
 \bibpunct[, ]{[}{]}{,}{n}{}{,}%
 \def\bibsep{\smallskipamount}%

\renewcommand{\cite}{\citet*} 
\setcitestyle{authoryear,round} 

\usepackage[british]{babel} 
\usepackage[pdfstartview=FitH,bookmarksnumbered=true,bookmarksopen=true,
            colorlinks=true,pdfborder={0 0 1},citecolor=blue,
            linkcolor=blue,urlcolor=blue]{hyperref}

\usepackage[ruled,vlined]{algorithm2e}

\usepackage{fancyhdr}      
\usepackage{indentfirst}   
\usepackage{enumerate}     
\usepackage{setspace}      
\usepackage[hang,flushmargin]{footmisc} 

\usepackage[misc]{ifsym}   

\usepackage[top=3.2cm, bottom=3.3cm, left=2.8cm, right=2.8cm]{geometry}

\usepackage[right,pagewise,mathlines]{lineno} 
\usepackage{comment}  
\usepackage{url}      

\DeclareMathOperator{\diag}{diag}    

\def\d{\mathrm{d}} 

\newcommand{\E}{\mathbb{E}}
\newcommand{\mP}{\mathbb{P}}        
\newcommand{\mR}{\mathbb{R}}        

\newcommand{\mb}[1]{\mathbb{#1}}    
\newcommand{\mbf}[1]{\mathbf{#1}}   
\newcommand{\mc}[1]{\mathcal{#1}}   

\newcommand{\vy}{{\mathbf{y}}}
\newcommand{\vx}{{\mathbf{x}}}
\newcommand{\vsigma}{\boldsymbol{\sigma}}
\newcommand{\vbeta}{\boldsymbol{\beta}}
\newcommand{\valpha}{\boldsymbol{\alpha}}
\newcommand{\vdelta}{\boldsymbol{\delta}}
\newcommand{\vepsilon}{\boldsymbol{\varepsilon}}
\newcommand{\vvarphi}{\boldsymbol{\varphi}}
\newcommand{\vb}{{\boldsymbol{b}}}
\newcommand{\vW}{{\mathbf{W}}}
\newcommand{\vS}{{\mathbf{S}}}
\newcommand{\vX}{{\mathbf{X}}}
\newcommand{\vY}{{\mathbf{Y}}}
\newcommand{\vz}{{\mathbf{z}}}
\newcommand{\vK}{{\mathbf{K}}}

\newcommand{\htB}{\hat{\boldsymbol{b}}}
\newcommand{\htBi}{\hat{b}^i}
\newcommand{\htW}{\widehat{W}}
\newcommand{\htwb}{\vW^{\hat{\vb}}}
\newcommand{\mpB}{\mathbb{P}^{\hat{\vb}}}

\newcommand{\pt}{\partial}          
\newcommand{\nm}[1]{\left\| #1 \right\|}  

\newcommand{\lt}{\left}
\newcommand{\rt}{\right}

\DeclareMathOperator{\Law}{Law}

\renewcommand{\ge}{\geqslant}
\renewcommand{\le}{\leqslant}

\renewcommand{\leq}{\leqslant}
\renewcommand{\epsilon}{\varepsilon}

\theoremstyle{plain}
\newtheorem{theorem}{Theorem}[section]

\newtheorem{lemma}{Lemma}[section]
\newtheorem{proposition}{Proposition}[section]

\theoremstyle{definition}

\newtheorem{example}{Example}[section]

\theoremstyle{remark}
\newtheorem{remark}{Remark}[section] 

\theoremstyle{definition}

\numberwithin{equation}{section} 

\begin{document} 
	
\title{
Robust Bayesian Portfolio Optimization with Discrepancy-based Posterior Ambiguity

}

\author{
Zongxia Liang\thanks{\scriptsize Department of Mathematical Sciences, Tsinghua University, Beijing, China. Email: \texttt{liangzongxia@mail.tsinghua.edu.cn}}
\and Yang Liu\thanks{\scriptsize School of Science and Engineering, The Chinese University of Hong Kong (Shenzhen), Shenzhen, China. Email: \texttt{yangliu16@cuhk.edu.cn}} 
\and Xingjian Ma\thanks{\scriptsize Corresponding author. Department of Mathematical Sciences, Tsinghua University, Beijing, China. Email: \texttt{mxj22@mails.tsinghua.edu.cn}}
}

\date{}

\maketitle

\begin{abstract}

We study a continuous-time robust Bayesian portfolio optimization problem under drift uncertainty of risky assets. The investor learns unknown asset drifts through Bayesian filtering while considering uncertainty around posterior estimates via discrepancy-based ambiguity sets, including Wasserstein and $L^p$ distances. To address the resulting time inconsistency, we introduce a feedback-type ambiguity framework that reformulates ambiguity conditionally on observable states. This leads to a modified Hamilton--Jacobi--Bellman--Isaacs (HJBI) equation characterizing the value function and the optimal strategy. For a semi-explicit solution example, we use the exponential utility to derive a reduced semilinear parabolic PDE and establish existence of classical solutions via a verification theorem. 

\noindent \textbf{Keywords}: Robust control, Bayesian learning, Stochastic control, Hamilton--Jacobi--Bellman--Isaacs equation, Discrepancy-based ambiguity, Wasserstein distance. 

\noindent \textbf{MSC codes}: 91G10, 93E20, 49L20. 
\end{abstract}

\section{Introduction}\label{sec:intro}

Continuous-time portfolio optimization has been a central topic in stochastic control and mathematical finance since the pioneering works of \cite{merton1969lifetime,merton1975optimum}. In the classical formulation, an investor dynamically allocates wealth among financial assets so as to maximize the expected utility of the terminal wealth, and the market coefficients, especially the drift vector of risky assets, are assumed to be known. This assumption is restrictive in practice. Asset return drifts are difficult to estimate and are typically not directly observable from market data. This has motivated a large stream of literature on portfolio optimization under partial information, where Bayesian learning and stochastic filtering are used to construct and update the Bayesian estimator of the unknown drift from observed asset prices; see, e.g., \cite{KZ01}, \cite{BGP19}, \cite{DNP19} and \cite{ZZ22}.

Bayesian learning provides a natural mechanism for incorporating information over time. However, the resulting Bayesian estimator process remains sensitive to the prior distribution, the filtering model and the assumed market specification. Prior misspecification, filtering errors or structural uncertainty may lead to overconfident posterior Bayesian estimators and fragile portfolio decisions. Thus, although Bayesian learning addresses the statistical uncertainty arising from partial observation, it does not by itself provide protection against ambiguity surrounding the Bayesian estimator.

A natural way to address this issue is to incorporate robustness into the Bayesian learning framework. Robust control has become an important tool for dynamic decision making under model uncertainty; see, for example, \cite{hansen2001robust}, \cite{uppal2003model}, \cite{garlappi2007portfolio}, \cite{jin2015continuous}, \cite{NN18} and \cite{Pesenti2023}. More recently, distributionally robust optimization (DRO) has provided a flexible methodology for modeling ambiguity through discrepancy-based neighborhoods around reference distributions; see, e.g., \cite{delage2010distributionally}, \cite{goh2010distributionally}, \cite{wiesemann2013robust}, \cite{Lam2016}, \cite{BM19}, \cite{Blanchet2021WassersteinDRO}, \cite{rahimian2022frameworks}, \cite{gao2024wasserstein}, \cite{jiang2024data} and \cite{Kuhn2025}. Owing to its flexibility and strong out-of-sample performance, DRO has attracted substantial attention in optimization, operations research and machine learning. Its application to continuous-time portfolio optimization and mathematical finance, however, remains comparatively limited, especially in settings involving dynamic learning and partial information; only a few studies (see, e.g., \cite{bardakci2019distributionally,blanchet2022distributionally,fan2024distributionally}, \cite{blanchet2025bayesian}) have appeared in recent years. This observation motivates a robust Bayesian formulation in which ambiguity is imposed directly on the Bayesian estimator process generated by partial observations, rather than on a fixed parameter or a static reference distribution.


In this paper, we study a continuous-time robust Bayesian portfolio optimization problem under drift uncertainty. The investor observes asset prices, updates the Bayesian estimator of the unknown drift, and evaluates portfolio strategies against a family of relaxed Bayesian estimators lying in discrepancy-based ambiguity sets around the posterior Bayesian estimator. The ambiguity level may vary over time and with the information generated by the market, thereby allowing the investor's confidence in the Bayesian estimator process to evolve dynamically. This formulation accommodates a broad class of ambiguity specifications, including Wasserstein, $L^p$, pathwise and multiple-constraint formulations.

The main difficulty is that Bayesian-estimator-level ambiguity naturally induces time inconsistency. Because the posterior Bayesian estimator and the ambiguity tolerance level evolve over time, a worst-case relaxed Bayesian estimator selected at one time need not remain worst-case at later times. Consequently, the standard dynamic programming principle cannot be applied directly. To overcome this difficulty, we introduce a feedback-type ambiguity framework in which admissible relaxed Bayesian estimators are reformulated conditionally on observable state variables. This reformulation preserves the robust Bayesian interpretation while restoring analytical tractability. It leads to a modified Hamilton--Jacobi--Bellman--Isaacs (HJBI) equation, through which we characterize the value function and the optimal feedback strategy by PDE methods, relying on classical second-order parabolic PDE theory; see, e.g., \cite{ladyzhenskaia1968linear}, \cite{krylov1987nonlinear} and \cite{lieberman1996second}.

The contributions of this paper are threefold. First, we formulate a continuous-time robust Bayesian portfolio optimization framework under drift uncertainty in which discrepancy-based ambiguity sets are imposed directly around the posterior Bayesian estimator. This provides a Bayesian-estimator-level robustification of the learning procedure, rather than a static or prior-level perturbation of model parameters. Second, we develop a feedback-type ambiguity reformulation that resolves the time inconsistency generated by the dynamically evolving Bayesian estimator process and its associated ambiguity sets. The resulting modified HJBI equation yields a characterization of the value function and the optimal robust feedback strategy. Third, we establish existence, verification and extension results for a broad class of ambiguity structures, including Wasserstein, $L^p$, pathwise and multiple-constraint ambiguity sets, all within a unified continuous-time PDE framework.

Our approach is connected to several strands of literature. From the modeling perspective, it is related to Bayesian ambiguity and time-inconsistent portfolio models such as \cite{BM14}, \cite{BKM17}, \cite{GLX25} and \cite{Han2025}. Unlike preference-based or equilibrium formulations, our ambiguity is introduced through measure-based discrepancy neighborhoods of the posterior Bayesian estimator. From the methodological perspective, our work is related to robust portfolio optimization under parameter uncertainty, especially \cite{pham17}, \cite{pham19} and \cite{pham22}. However, the presence of Bayesian filtering enlarges the state space and substantially complicates the associated HJBI equation, making explicit solutions generally unavailable. The feedback-type reformulation developed in this paper is designed precisely to handle this interaction between learning, robustness and dynamic optimization. Meanwhile, a related strand of literature studies distributionally robust control (DRC) or distributionally robust Markov decision processes (DRMDP), which typically adopt time-rectangular ambiguity sets that preserve the dynamic programming structure; see, e.g., \cite{iyengar2005robust}, \cite{nilim2005robust}, \cite{hansen2008robustness}, \cite{xu2010distributionally}, \cite{wiesemann2013robust}, \cite{BKN21} and \cite{lu2024distributionally}.  Our posterior Bayesian estimator ambiguity differs from this standard rectangular formulation, because it evolves endogenously with the Bayesian estimator and therefore creates a time-consistency issue.

Finally, our framework should be distinguished from existing robust Bayesian portfolio models. Earlier studies mainly combine posterior Bayesian learning with entropy-based model misspecification penalties or static parameter uncertainty sets, often in discrete-time or single-period settings; see, e.g., \cite{anderson2016robust}, \cite{SZL23} and \cite{quimbayo2025robust}. The closest continuous-time contribution is \cite{blanchet2025bayesian}, which imposes optimal-transport-type ambiguity on the drift prior. Their ambiguity is imposed ex ante at the prior level and the resulting strategy is interpreted in a pre-committed sense, whereas we impose discrepancy-based ambiguity directly around the posterior Bayesian estimator and allow it to evolve with the Bayesian estimator process and observable market information. Methodologically, their tractability relies on a minimax reduction over candidate priors, while our tractability comes from the feedback-type ambiguity reformulation and the modified HJBI equation. Moreover, while \cite{blanchet2025bayesian} focuses on prior-level optimal-transport-type uncertainty, our framework covers Wasserstein, $L^p$, pathwise and multiple-constraint posterior-estimator ambiguity structures in a unified continuous-time PDE setting.

The remainder of this paper is organized as follows. Section~\ref{sec:prob_setup} introduces the financial market, the Bayesian filtering framework, the ambiguity specifications and the robust Bayesian optimization problem. Section~\ref{sec:HJB} develops the martingale optimality principle, derives the modified HJBI equation, obtains the corresponding semi-explicit solution through a reduced parabolic PDE and characterizes the optimal feedback controls. Section~\ref{sec:verification} proves the verification theorem. Section~\ref{sec:example_and_extensions} verifies the required conditions for several representative ambiguity sets and discusses extensions. Section~\ref{sec:conclusion} concludes.

\section{Robust Bayesian Optimization Problem}\label{sec:prob_setup}


In this section, we formulate a robust Bayesian optimization problem by introducing uncertainty around the posterior Bayesian estimator through the discrepancy-constrained ambiguity set. We first present the market model and the Bayesian learning framework, and then introduce the formulation of the ambiguity set and the associated robust control problem. This formulation provides the foundation for the HJBI analysis and verification results developed in the subsequent subsections.


\subsection{Financial Market and Bayesian Estimation}

We consider a financial market under partial information, where the drift of risky assets is unknown and must be inferred from observed price dynamics. In this setting, investors continuously update their beliefs about the unobservable drift through Bayesian filtering as new market information becomes available.

Let $(\Omega, \mathcal{F}, \{\mathcal{F}_t\}_{t \in [0,T]}, \mathbb{P})$ be a filtered complete probability space, where the constant $T > 0$ is the time horizon and the filtration $\{\mathcal{F}_t\}_{t \in [0,T]}$ satisfies the usual conditions and represents the whole information of the financial market. Assume that $\{\vW_t = (W_t^1, \cdots, W_t^d)^\top  \}_{t \in [0,T]}$  is a $d$-dimensional Brownian motion adapted to $\{\mathcal{F}_t \}_{t \in [0,T]}$, with correlation structure given by
$\d \langle W^i, W^j \rangle_t = \rho^{ij} \,\mathrm{d} t$ for all $i,j \in \{1,\cdots,d\}$.

The financial market consists of two types of assets: one bond (risk-free asset) and $d$ stocks (risky assets). The risk-free interest rate is denoted by $r$. The price of the assets $\{ \vS_t = (S_t^1, \cdots, S_t^d)^\top \}_{t\in [0,T]}$ has the classic log-normal dynamic:
\begin{equation*}
\d \vS_t = \diag (\vS_t)(\vb \,\mathrm{d} t + \vsigma \odot \d \vW_t),
\end{equation*}
where $\odot$ denotes the element-wise multiplication of vectors. Here, the constant volatility vector $\vsigma = (\sigma^1, \cdots, \sigma^d)^\top$ satisfies $\sigma^i > 0$, $\forall i \in \{1,\cdots,d\}$, and the random drift vector $\mbf{b} = (b^1, \cdots, b^d)^\top$ is unknown. Throughout the paper, we denote by 
$\rho = (\rho^{ij})_{1\leq i,j \leq d}$ and 
$\Sigma = (\rho^{ij} \sigma^i \sigma^j)_{1\leq i,j \leq d}$ 
the correlation and covariance matrices associated with the dynamics of prices respectively.

\begin{remark}
  In classic literature (see, e.g., \cite{KZ01}, \cite{BGP19}), without loss of generality, we can assume that $r=0$ so that the risk-free asset becomes a constant. As our modeling is slightly different, we will apply this assumption later to make the fact clear for the readers.
\end{remark}

In practice, the investor does not have access to the whole information but observes the evolutions of the asset prices. Therefore the filtration of her admissible information is the natural filtration $\mb{F}^\vS:=\{\mathcal{F}_t^\vS\}_{t \in [0,T]}$ generated by the stock price process $\vS = \{ \vS_t \}_{t \in [0,T]}$. 
Consider the Bayesian estimator $\vbeta = \{\vbeta_t\}_{t \in [0,T]}$ given by
\begin{equation*}
\vbeta_t \triangleq \mathbb{E}[\mbf{b} \mid \mathcal{F}_t^\vS], \quad t \in [0,T].
\end{equation*}
Following \cite{BGP19}, we assume that the prior distribution of $\vb$, denoted by $\mu_0$, satisfies the following condition: there exists $\eta > 0$ such that
\begin{equation}\label{eq:sub_gaussian_condition}
\mathbb{E}\bigl[e^{\eta \|\mu_0\|^2}\bigr] = \int_{\vz \in \mR^d} e^{\eta \|\vz\|^2} \, \mu_0(\d \vz) < \infty.
\end{equation}

Then, we can represent the dynamics with respect to the known information $\{\mathcal{F}_t^\vS\}_{t \in [0,T]}$.
To represent the Bayesian estimator $\vbeta$ in terms of observable information, we introduce the process $\vY = \{ \vY_t \}_{t \in [0,T]}$ which is defined by $Y_t^i = \log S_t^i$, $\forall i \in \{1,\cdots,d\}$, $t \in [0,T]$. Then, the filtration $\mb{F}^\vS$ can be equivalently generated by $\vY$. By \cite[Theorem 1]{BGP19}, $\vbeta$ can be represented by
\begin{equation*}
\vbeta_t = \vvarphi(t, \vY_t), \quad t \in [0,T],
\end{equation*}
where
\begin{equation}\label{eq:varphi_formation}
    \begin{aligned}
       \vvarphi(t, \vY_t) & = \frac{\Sigma  \nabla_{\vy} F_1(t, \vY_t)}{F_1(t, \vY_t)}+r\vec{1}, \\
      F_1(t, \vy) & = \int_{\mR^d} \exp \left(
(\vz-r\vec{1})^\top \Sigma^{-1}
\left[
\vy-\vY_0+\left(-r\vec{1}+\frac{1}{2}\vsigma \odot \vsigma\right)t
-\frac{t}{2}(\vz-r\vec{1})
\right]
\right)
\, \mu_0(\d \vz).
    \end{aligned}
\end{equation}
Let us define the innovation process $\widehat{\vW} =\{ \widehat{\vW}_t \}_{t \in [0,T]}$ by
\begin{equation*}
\htW^i_t = W_t^i + \int_0^t \frac{b^i - \beta^i_s}{\sigma^i} \,\mathrm{d} s, \quad \forall i \in \{1,\cdots,d\}, \ \forall t \in [0,T].
\end{equation*}
As shown in \cite[Proposition 2]{BGP19}, $\{\widehat{\vW}_t \}_{t \in [0,T]}$ is a $d$-dimensional Brownian motion adapted to $\{\mc{F}^\vS_t \}_{t \in [0,T]}$ under probability $\mathbb{P}$, with the same correlation structure as $\{\vW_t \}_{t \in [0,T]}$.
Then, in terms of $\widehat{\vW}$, the stock price evolves according to the following stochastic differential equation (SDE):
\begin{equation*}
\d \vS_t = \diag{(\vS_t)} \lt( \vbeta_t  \,\mathrm{d} t + \vsigma \odot \d \widehat{\vW}_t \rt).
\end{equation*}
Meanwhile, 
\begin{equation*}
\d \vY_t = \lt(\vbeta_t - \frac{\vsigma \odot \vsigma}{2}\rt)\,\mathrm{d} t + \vsigma \odot\d \widehat{\vW}_t.
\end{equation*}
Here, $\{\widehat{\vW}_t \}_{t \in [0,T]}$ is called the innovation process. Expressing the dynamics in terms of the innovation process allows all coefficients to become adapted to the observable filtration $\mc{F}^\vS$, thereby reformulating the stock price dynamics entirely in terms of known information available to the investor.

\subsection{Ambiguity Setup}\label{subsec:ambiguity_setup}

Although Bayesian learning provides dynamic estimates of unknown market parameters, the resulting posterior process may still be affected by estimation errors, prior misspecification, or structural uncertainty. 
Therefore, we define a general class of ambiguity sets through discrepancy constraints with time-varying tolerance levels. This robust formulation enables the investor to incorporate uncertainty not only in the underlying market model, but also in the reliability of the Bayesian estimator itself.

First, we define the relaxed Bayesian estimator $\htB=\{ \htB_t = \htB(t,\vY_t) \}_{t\in[0,T]}$ as an $\mR^d$-valued process in the admissible ambiguity set $\mc{B}$, which allows the time-varying ambiguity tolerance level as follows:
\begin{equation}\label{eq:bayesian_ambiguity_set_general}
\mc{B} :=
\left\{
\htB = (\htB_t)_{t \in [0,T]} :
\begin{array}{l}
\htB_t = \htB(t,\vY_t),\\
\htB : [0,T]\times \mR^d \to \mR^d \text{ is Borel measurable},\\
\mc{D}(\htB_t,\vbeta_t) \le \epsilon(t),\quad \forall t \in [0,T]
\end{array}
\right\},
\end{equation}
where $\mc{D}$ is a discrepancy functional measuring the distance between two random variables. By the definition, we can see that $\htB$ is $\mb{F}^\vS$-progressively measurable.
Meanwhile, as the investor will be more confident when time goes on, we suppose that the time-varying tolerance level function $\epsilon: [0,T] \longrightarrow \mR_+$ is decreasing, and is $\alpha$-H\"older continuous for some $\alpha \in (0,1)$, denoted as $\epsilon(\cdot) \in C^{\alpha}[0,T]$.

\begin{remark}
Such requirement on the time-varying tolerance level $\epsilon(\cdot)$ is mild, as any decreasing polynomial function on $[0,T]$ satisfies this condition.
The monotonicity assumption on $\epsilon(\cdot)$ mainly becomes relevant in the multiple-constraint setting developed later, where componentwise ordering is needed to preserve the structural properties of the ambiguity set. In the single-constraint case, this condition can be relaxed to boundedness; see Section~\ref{sec:example_and_extensions} for the details.
Meanwhile, the H\"older continuity assumption is imposed primarily for analytical purposes, ensuring sufficient regularity of the associated source term in the reduced PDE~\eqref{eq:reduced_semi_explicit_HJB_equation} and thereby supporting well-posedness through classic parabolic PDE theory.
\end{remark}


Then, we present several representative examples of admissible ambiguity sets. We begin with an ambiguity set based on distributional discrepancy, which serves as the primary motivation of our framework, and then consider other discrepancy-based formulations that also fit into the general structure.

\begin{example}[Distributional discrepancy: p-Wasserstein distance]\label{ex:L2_distribution}
Let $\mc{D}$ be a discrepancy functional defined on probability measures, given by
\begin{equation*}
\mc{D}(\vX,\vY) := W_{p}\big( \Law(\vX), \Law(\vY) \big),
\end{equation*}
where $W_p$ is the p-Wasserstein distance induced by a norm $\|\cdot\|$ on $\mR^d$, defined by
\begin{equation*}
W_p(\mu,\nu)
:=
\left(
\inf_{\pi\in\Pi(\mu,\nu)}
\int_{\mR^d\times\mR^d}
\nm{\vx-\vy}^p\,\pi(\d \vx,\d \vy)
\right)^{1/p},
\end{equation*}
where $
\Pi(\mu,\nu)
:=
\left\{
\pi\in\mathcal{P}(\mR^d\times\mR^d)
:
\pi(\cdot,\mR^d)=\mu,\;
\pi(\mR^d,\cdot)=\nu
\right\} $
is the set of all couplings of $\mu$ and $\nu$.
Then the ambiguity set becomes
\begin{equation*}
\mc{B} =
\left\{
\htB :
\htB_t = \htB(t,\vY_t), \,
W_p\big( \Law(\htB_t), \Law(\vbeta_t) \big) \le \epsilon(t),\ \forall t \in [0,T]
\right\}.
\end{equation*}
\end{example}

\begin{example}[$L^p$-discrepancy]\label{ex:L2_expectation}
Let $\mc{D}$ be defined by
\begin{equation*}
\mc{D}(\vX,\vY) := \big( \E[ \|\vX - \vY\|^p ] \big)^{1/p}, \, p>1,
\end{equation*}
where $\|\cdot\|$ can be chosen as any norm on $\mR^d$. Then the ambiguity set becomes
\begin{equation*}
\mc{B} =
\left\{
\htB :
\htB_t = \htB(t,\vY_t), \,
\big( \E[ \|\htB_t - \vbeta_t\|^p ] \big)^{1/p} \le \epsilon(t),\ \forall t \in [0,T]
\right\}.
\end{equation*}
\end{example}
Actually, the ambiguity set can be chosen from a larger family. When we replace the single time-varying constraint $\mc{D}(\htB_t,\vbeta_t) \le \epsilon(t)$ in \eqref{eq:bayesian_ambiguity_set_general} by sample-path discrepancy or multiple discrepancy constraints as follows, our model incorporates these cases and our main results still remain true. We will show such extensions in Section~\ref{sec:example_and_extensions}.

\begin{example}[Sample-path discrepancy]\label{ex:L2_sample_path}
Let $\mc{D}$ be defined by
\begin{equation*}
\mc{D}(\vx,\vy) := \|\vx - \vy\|,
\end{equation*}
where $\|\cdot\|$ can be chosen as any norm on $\mR^d$.  Then we may define the ambiguity set under sample-path discrepancy as follows:
\begin{equation*}
\mc{B} =
\left\{
\htB :
\htB \text{ is } \mR^d\text{-valued, } \mb{F}^{\vS}\text{-progressively measurable, }
\|\htB_t(\omega) - \vbeta_t(\omega)\| \le \epsilon(t),\ \forall \omega \in \Omega,\, t \in [0,T]
\right\}.
\end{equation*}
\end{example}

\begin{example}[Multiple discrepancy constraints]\label{ex:multi_constraints}
Let $\mc{D}_i, \, i=1,2,\cdots,n$ be several discrepancy functionals, for example, the $L^p$-discrepancy or the p-Wasserstein distance defined above. Then we may define the ambiguity set under multiple discrepancy constraints by
\begin{equation*}
\mc{B} =
\left\{
\htB :
\htB_t = \htB(t,\vY_t),
\mc{D}_i(\htB_t,\vbeta_t) \le \epsilon_{i}(t), \, i=1,2,\cdots,n, \, \forall t \in [0,T]
\right\}.
\end{equation*}
\end{example}

\subsection{Dynamics of the Stock Price}

We note that, even when ambiguity is incorporated into the Bayesian estimation framework, the investor’s available information is still generated solely by observations of the stock price process $\vS$. Therefore, it is necessary to characterize the dynamic of $\vS$ under the ambiguity formulation in a way that preserves the realized price path. This can be achieved through an appropriate Girsanov transformation.

By Girsanov's theorem (see, e.g., \cite{karatzas1998brownian}), consider the process $\vW^{\htB}:=\{\vW^{\htB}_t\}_{t\in [0,T]}$ given by
\begin{equation*}
W^{\htB,i}_t = W_t^i + \int_0^t \frac{b^i - \htBi_s}{\sigma^i} \,\mathrm{d} s = \htW_t^i + \int_0^t \frac{\beta_t^i - \htBi_s}{\sigma^i} \,\mathrm{d} s, \quad \forall i \in \{1,\cdots,d\}, \ \forall t \in [0,T],
\end{equation*}
There exists a corresponding measure $\mb{P}^{\htB}$ such that the process $\{\vW^{\htB}_t, \mc{F}^{S}_t\}_{t \in [0,T]}$ is a Brownian motion under $(\Omega, \mb{F}^{S},\mb{P}^{\htB})$, with the same correlation structure as $\{ \vW_t \}_{t \in [0,T]}$.

Then, the dynamic of the stock price, under the ambiguity formulation in the Bayesian estimation, can be characterized through a change of probability measures as follows:
\begin{equation*}
  \begin{aligned}
      \d \vS_t
 & = \diag{(\vS_t)} \lt( \vbeta_t  \,\mathrm{d} t + \vsigma \odot \d \widehat{\vW}_t \rt),
  \qquad \mP \text{-a.s.} ,\\
  & = \diag{(\vS_t)} \lt( \htB_t \,\mathrm{d} t + \vsigma \odot \d \htwb_t \rt),
  \qquad \mP^{\htB} \text{-a.s.} ,
  \end{aligned}
\end{equation*}
which provides an equivalent probabilistic representation of model uncertainty while preserving the observable informational structure generated by the observed stock price process.

\subsection{Robust Wealth Dynamics and Optimization Problem}

With the robust Bayesian framework established above, we now formulate the investor’s optimal control problem. We first introduce the class of admissible trading strategies, which are required to be adapted to the observable filtration. Let the control strategy ${\valpha}=\{ {\valpha}_t \}_{t\in[0,T]}$ be an $\mR^d$-valued process, where the $i$-th component $\alpha_t^i$ represents the amount of wealth invested in the risky asset $i$ at time $t$.  
We define the admissible strategy set as
\begin{equation*}
  \mathcal{A}
  :=
  \left\{
    {\valpha} :
    {\valpha} \text{ is } \mb{F}^\vS\text{-progressively measurable and }
    \mathbb{E}\!\int_0^T \|{\valpha}_t\|^2 dt < \infty
  \right\}.
\end{equation*}
Given an initial capital $x_0>0$, under a candidate relaxed Bayesian estimator $\htB$, the corresponding wealth process $X^{\valpha}=\{ X_t^{\valpha} \}_{t\in[0,T]}$ evolves according to
\begin{equation}\label{eq:wealth_dynamic}
\left\{
  \begin{aligned}
      \d X_t^{\valpha}
 & = r X_t^{\valpha} \d t + {\valpha}_t^\top \lt( (\vbeta_t - r \vec{1}) \,\mathrm{d} t + \vsigma \odot \d \widehat{\vW}_t\rt),
  \qquad \mP \text{-a.s.} ,\\
  & = r X_t^{\valpha} \d t + {\valpha}_t^\top \lt( (\htB_t - r \vec{1}) \,\mathrm{d} t + \vsigma \odot \d \htwb_t\rt),
  \qquad \mP^{\htB} \text{-a.s.} , \\
  \qquad
  X_0^{\valpha} & = x_0.
  \end{aligned}
\right.
\end{equation}
For clarity, we also denote the state process $X_t^{\valpha}$ under probability $\mP^{\htB}$ as $X_t^{{\valpha},\htB}$.

Based on these dynamics, the investor seeks to maximize the expected utility of the terminal wealth while considering the worst case within the ambiguity set defined above. This leads naturally to the robust Bayesian optimization problem~\eqref{eq:robust_mpr_objective} as follows, which is a stochastic robust control problem, and the optimal control is chosen against the worst-case relaxed Bayesian estimator:


\begin{equation} \label{eq:robust_mpr_objective}
\begin{aligned}
    &V_0  := 
\sup_{{\valpha} \in \mathcal{A}}
\;
\inf_{ \htB \in \mc{B}}
\; J({\valpha}, \htB) = \sup_{{\valpha} \in \mathcal{A}}
\;
\inf_{ \htB \in \mc{B}}
\;  \mathbb{E}^{\mP^{\htB}} \left[
U\!\left(X_T^{\valpha}\right)
\right], \\
&J({\valpha}, \htB) :=
\mathbb{E}^{\mP^{\htB}}
\left[
U\!\left(X_T^{\valpha}\right)
\right],
\end{aligned}
\end{equation}
where $U : \mR \to \mR$ is chosen as the CARA utility $U(x)=-e^{-\gamma x}, \, \gamma>0$.
It is important to note that the optimization in \eqref{eq:robust_mpr_objective} is not a pointwise max-min problem. Both the portfolio strategy $\valpha$ and the relaxed Bayesian estimator $\htB$ are stochastic processes, so the supremum and infimum are taken over infinite-dimensional admissible classes; in particular, a pointwise worst-case choice need not directly define an admissible relaxed Bayesian estimator process. This point will be discussed in more detail in Section~\ref{subsec:HJBI_construction}.

\begin{remark}
In this paper, we primarily adopt the CARA utility to ensure that the associated HJBI equation admits a classical solution, and to present the robust Bayesian structure and the corresponding feedback controls in a relatively clear form.

For other utility specifications, such as the CRRA utility, much of the subsequent analytical framework is expected to remain applicable, although additional technical difficulties arise. Specifically, the construction of the HJBI equation and a similar reduction procedure can still be obtained through suitable structural transformations. However, in such cases, the resulting equation typically admits only a viscosity solution rather than a classical solution. Consequently, the classic verification arguments no longer apply directly. Moreover, under the viscosity framework, first-order derivatives or generalized subdifferentials are not necessarily uniquely defined, making the corresponding optimal feedback strategy potentially ill-posed. 
\end{remark}

\begin{remark}\label{rmk:robust_measure_clarify}
Under the robust formulation, the state process $X$ and the control ${\valpha}$ appear to depend on the choice of the relaxed Bayesian estimator $\htB$. However, the change of drift is implemented through a change of probability measure. More precisely, different choices of $\htB$ correspond to different equivalent probability measures on the same canonical space. 
Therefore, the underlying sample paths of $X$ and ${\valpha}$ remain unchanged. In particular, the dynamics are modified in law, but not in a pathwise sense.
\end{remark}

\begin{remark}\label{rmk:initial_setup}
    As mentioned above, without loss of generality, we assume that $r=0$. Otherwise, the case with a nonzero interest rate can be reduced to $r=0$ by applying suitable linear transformations to drift $\vb$, Bayesian estimator $\vbeta$, and the wealth process $X^{\valpha}$.
    Then, the dynamic of the state process reduces to
    \begin{equation}\label{eq:state_process_r_is_0}
  \begin{aligned}
      \d X_t^{\valpha}
 & = {\valpha}_t^\top \lt( \vbeta_t \,\mathrm{d} t + \vsigma \odot \d \widehat{\vW}_t\rt),
  \qquad \mP \text{-a.s.}, \\
  & = {\valpha}_t^\top \lt( \htB_t  \,\mathrm{d} t + \vsigma \odot \d \htwb_t\rt),
  \qquad \mP^{\htB} \text{-a.s.} .
  \end{aligned}
\end{equation}
And the dynamic of $\vY$ becomes
\begin{equation*}
    \d \vY_t = \lt(\htB_t - \frac{\vsigma \odot \vsigma}{2}\rt)\,\mathrm{d} t + \vsigma \odot \d \htwb_t.
\end{equation*}
Similarly, without loss of generality, we assume the initial price of the stocks as $\vS_0 =\vec{1}$, and then have $\vY_0= \vec{0}$.
\end{remark}


\section{Hamilton–Jacobi–Bellman–Isaacs Equations}\label{sec:HJB}

In this section, based on the martingale optimality principle, we derive the Hamilton--Jacobi--Bellman--Isaacs (HJBI) equation associated with the robust Bayesian optimization problem. To address the structural complexity of the original problem, we introduce a modified HJBI equation, which provides a clearer characterization while preserving the essential features of the robust control framework.

\subsection{Martingale Optimality Principle}\label{subsec:mart_opt}

First, we develop the martingale optimality principle for the robust Bayesian optimization problem. This principle plays a role similar to the dynamic programming principle in classic stochastic control, providing the main tool for characterizing the value function and the optimal strategy. We now present the martingale optimality principle as follows.

\begin{theorem}[Martingale optimality principle]\label{thm:mart_opt_principle}

Let $\{V^{{\valpha},\htB}: {\valpha} \in \mathcal{A},\, \htB \in \mc{B}\}$ be a family of stochastic processes in the form
$V_t^{{\valpha},\htB} = v(t,X_t^{{\valpha},\htB},\vY_t)$
for some real-valued measurable function $v$ on $[0,T]\times \mR \times \mR^d$ satisfying
\begin{itemize}
\item[(i)]
\begin{equation*}
v(T,x,\vy) = U(x), \qquad \forall \, (x,\vy)\in \mR \times \mR^d.
\end{equation*}

\item[(ii)]
The process $\lt\{V_t^{{\valpha},\htB^*}\rt\}_{t \in [0,T]}$  is a supermartingale under $\mb{P}^{\htB^*}$ for all ${\valpha} \in \mc{A}$ and some $\htB^* \in \mc{B}$.

\item[(iii)]
The process $\lt\{V_t^{{\valpha}^*,\htB}\rt\}_{t \in [0,T]}$  is a submartingale under $\mb{P}^{\htB}$ for some ${\valpha}^* \in \mc{A}$ and all $\htB \in \mc{B}$.
\end{itemize}
Then, $({\valpha}^*,\htB^*)$ is an optimizer for the robust Bayesian problem \eqref{eq:robust_mpr_objective} with optimal value function
\begin{equation}\label{eq:thm1-2}
V_0 = V(0,x_0,0) =\sup_{{\valpha} \in \mathcal{A}} \inf_{\htB \in \mc{B}} J({\valpha}, \htB)= \inf_{\htB \in \mc{B}} \sup_{{\valpha} \in \mathcal{A}} J({\valpha}, \htB) = J({\valpha}^*,\htB^*)=\E^{\mP^{\htB^*}}\lt[V_T^{{\valpha}^*,\htB^*}\rt].
\end{equation}
\end{theorem}
\begin{proof}
   Observing that $V_0^{{\valpha},\htB} = v(0,x_0,0)$ for all ${\valpha} \in \mc{A}$ and $\htB \in \mc{B}$ and using Condition (i), we have $\E^{\mpB}\lt[V_T^{{\valpha},\htB}\rt] = J({\valpha}, \htB)$. Then, using Condition (ii), we obtain that, for some $\htB^* \in \mc{B}$,
   \begin{equation*}
       v(0,x_0,0) = \E^{\mb{P}^{\htB^*}}\lt[V_0^{{\valpha},\htB^*}\rt] \ge \E^{\mb{P}^{\htB^*}}\lt[V_T^{{\valpha},\htB^*}\rt] = J({\valpha},\htB^*), \; \forall {\valpha} \in \mc{A}.
   \end{equation*}
   Then
   \begin{equation*}
       v(0,x_0,0) \ge \sup_{{\valpha} \in \mc{A}}J({\valpha},\htB^*) \ge \inf_{\htB \in \mc{B}} \sup_{{\valpha} \in \mc{A}} J({\valpha},\htB).
   \end{equation*}
Similarly, based on  Condition (iii), we  obtain  $v(0,x_0,0) \le J({\valpha}^*,\htB), \; \forall \htB \in \mc{B}$. Therefore
   \begin{equation*}
       v(0,x_0,0) \le \inf_{\htB \in \mc{B}}J({\valpha}^*,\htB) \le \sup_{{\valpha} \in \mc{A}}\inf_{\htB \in \mc{B}}J({\valpha},\htB) = V_0.
   \end{equation*}
Meanwhile, the inequality
   \begin{equation*}
       \inf_{\htB \in \mc{B}} \sup_{{\valpha} \in \mc{A}} J({\valpha},\htB) \ge \sup_{{\valpha} \in \mc{A}}\inf_{\htB \in \mc{B}}J({\valpha},\htB)
   \end{equation*}
   always holds. Then, we obtain the required results by combining all above inequalities.
\end{proof}

Then, we derive the Hamiltonian.
In order to construct a process $V_t^{{\valpha},\htB} = v(t,X_t^{{\valpha},\htB},\vY_t)$ satisfying Conditions (i)--(iii) in Theorem~\ref{thm:mart_opt_principle}, we shall rely on Itô's formula. Assume that $v(t,x,\vy)$ is sufficiently smooth on $[0,T] \times \mR \times \mR^d$. For notational simplicity, write $X_t = X_t^{{\valpha},\htB}$. Then, applying Itô's formula to $v(t,X_t,\vY_t)$ and using the dynamics in Remark~\ref{rmk:initial_setup}, we have

\begin{equation}\label{eq:ito_results}
\begin{aligned}
\d v(t,X_t,\vY_t)
=&\Bigg[
\partial_t v(t,X_t,\vY_t)
+ \sum_{i=1}^d \partial_{x}v(t,X_t,\vY_t)\,\alpha_t^i \htBi_t
+ \sum_{i=1}^d \partial_{y_i}v(t,X_t,\vY_t)
\Big(\htBi_t-\tfrac12(\sigma^i)^2\Big)
\\[4pt]
&\quad
+ \tfrac12
\sum_{i,j=1}^d
\partial_{x x}^2 v(t,X_t,\vY_t)\,
\alpha_t^i \alpha_t^j\,\Sigma^{ij}
+ \tfrac12
\sum_{i,j=1}^d
\partial_{y_i y_j}^2 v(t,X_t,\vY_t)\,
\Sigma^{ij}
\\[4pt]
&\quad
+ \sum_{i,j=1}^d
\partial_{x y_j}^2 v(t,X_t,\vY_t)\,
\alpha_t^i\,\Sigma^{ij}
\Bigg]\d t
\\[6pt]
&\ + \sum_{i=1}^d
\Big(
\partial_{x}v(t,X_t,\vY_t)\,\alpha_t^i\sigma^i
+ \partial_{y_i}v(t,X_t,\vY_t)\,\sigma^i
\Big)\d W^{\htB,i}_t.
\end{aligned}
\end{equation}
The coefficient of the $\d t$ term in \eqref{eq:ito_results}, evaluated at $x,\vy$, is of the form
\begin{equation*}
    \partial_t v(t,x,\vy) + H(t,x,\vy,{\valpha}_t,\htB_t),
\end{equation*}
where $H(t,x,\vy,{\valpha}_t,\htB_t)$ is defined on $[0,T] \times \mR \times \mR^d \times \mR^d \times \mR^d $ by
\begin{equation*}
    \begin{aligned}
        H(t,x,\vy,{\valpha}_t,\htB_t) = &
\sum_{i=1}^d \partial_{x}v(t,x,\vy)\,\alpha_t^i \htBi_t
\!+\!\! \sum_{i=1}^d \partial_{y_i}v(t,x,\vy)
\big(\htBi_t\!-\!\tfrac12(\sigma^i)^2\big)
\!\!+ \tfrac12
\sum_{i,j=1}^d
\partial_{x x}^2 v(t,x,\vy)\,
\alpha_t^i \alpha_t^j\,\Sigma^{ij} \\
&+ \tfrac12
\sum_{i,j=1}^d
\partial_{y_i y_j}^2 v(t,x,\vy)\,
\Sigma^{ij}
+ \sum_{i,j=1}^d
\partial_{x y_j}^2 v(t,x,\vy)\,
\alpha_t^i\,\Sigma^{ij}.
    \end{aligned}
\end{equation*}
We rewrite the Hamiltonian in a vector form as follows:
\begin{equation}\label{eq:Hamiltonian_vector}
H(t,x,\vy,{\valpha}_t,\htB_t)
=
\frac12\,{\valpha}_t^\top A(t,x,\vy)\,{\valpha}_t
+ \vK(t,x,\vy,\htB_t)^\top {\valpha}_t
+ C(t,x,\vy,\htB_t),
\end{equation}
where the matrix function $A$, the vector function $\vK$, and the scalar function $C$ are given componentwise by
\begin{equation*}
\begin{aligned}
& A_{ij}(t,x,\vy):= \partial_{x x}^2 v(t,x,\vy)\,\Sigma^{ij},
\\
&K_i(t,x,\vy,\htB_t) := \partial_{x}v(t,x,\vy)\,\htBi_t
+ \sum_{j=1}^d \partial_{x y_j}^2 v(t,x,\vy)\,\Sigma^{ij},
\\&
C(t,x,\vy,\htB_t):= \sum_{i=1}^d
\partial_{y_i}v(t,x,\vy)
\Big(\htBi_t-\tfrac12(\sigma^i)^2\Big)
+ \tfrac12
\sum_{i,j=1}^d
\partial_{y_i y_j}^2 v(t,x,\vy)\,\Sigma^{ij}.
\end{aligned}
\end{equation*}

\begin{remark}
    Noting that the forms of $A(t,x,\vy), \, \vK(t,x,\vy,\htB_t), \, C(t,x,\vy,\htB_t)$ and the Hamiltonian $H$ all depend on the function $v$, we would append a superscript $v$ when this dependence needs to be emphasized.
\end{remark}

\subsection{Construction of HJBI Equations}\label{subsec:HJBI_construction}

This subsection constructs the HJBI equation corresponding to the robust Bayesian optimization problem as follows.

\subsubsection{Primal HJBI Equation}

At the start, we derive the primal form of the HJBI equation, which is the most direct formulation obtained from the martingale optimality principle. It follows the standard dynamic-programming intuition and reflects, in a formal way, the structure of the original robust Bayesian optimization problem:
\begin{equation}\tag{HJBI-P}\label{eq:primal_HJBI_equation}
    \begin{cases}
        & \pt_t v(t,x,\vy) + \inf\limits_{\htB \in \mc{B}} \sup\limits_{{\valpha} \in \mc{A}}  \E^{\mpB}[ H^v(t,x,\vy,{\valpha}_t,\htB_t) | X_t = x, \vY_t = \vy] = 0,  \\
        & \hfill (t,x,\vy)\in [0,T)\times \mR \times \mR^d, \\
        & v(T,x,\vy)=U(x),   \hfill  (x,\vy)\in \mR \times \mR^d.
    \end{cases}
\end{equation}
However, although this formulation may yield a function $v$ satisfying the martingale optimality principle at a formal way, it also reveals a fundamental obstacle caused by time inconsistency. More precisely, for each fixed $s \in [0,T)$, the optimization in the equation at time $s$ selects, if it exists, an optimal relaxed Bayesian estimator process (or possibly a family of optimal relaxed Bayesian estimator processes) $\htB^*(s) = \{\htB^*_t(s) \}_{t\in [0,T]} \in \mc{B}$ determined from the perspective of time $s$ but acts on the entire horizon $[0,T]$. Hence these optimal relaxed Bayesian estimator processes are global objects associated with different base times, rather than the time sections of a single admissible process. Consequently, there is in general no reason for the collection $\{\htB^*(s): s\in[0,T)\}$ to be mutually consistent in time, and it cannot be aggregated into a unified optimal relaxed Bayesian estimator $\htB^* = \{\htB^*_t \}_{t\in [0,T]} \in \mc{B}$ satisfying Condition~(ii) in Theorem~\ref{thm:mart_opt_principle}.


\subsubsection{Modified HJBI Equation}
To overcome time inconsistency occured in \eqref{eq:primal_HJBI_equation}, we introduce a feedback-type ambiguity set. This formulation restores time consistency at the level of admissible controls and leads to a modified HJBI equation, from which both the value function $v$ and the associated optimizers can be derived so as to satisfy Conditions~(i)--(iii) in Theorem~\ref{thm:mart_opt_principle}.

We consider a feedback-type ambiguity formulation: for each $(t,x,\vy) \in [0,T]\times \mR \times \mR^d$, we define the feedback-type ambiguity set at time $t$ under the state condition $(X_t = x, \vY_t = \vy)$ by
\begin{equation}\label{eq:feedback_B}
\mc{B}(t,\vy)  :=
\left\{
\vb \in \mR^d : \mc{D}(\vb,\vvarphi(t,\vy)) \le \epsilon(t)
\right\},
\end{equation}
where, by a slight abuse of notation, $\mc{D}(\cdot,\cdot)$ also denotes the restriction of the discrepancy functional in \eqref{eq:bayesian_ambiguity_set_general} on two constant-valued random variables. More precisely, for $\vx,\vy \in \mR^d$, $\mc{D}(\vx,\vy)$ is defined as the discrepancy between the two constant-valued random variables taking values $\vx$ and $\vy$, respectively.
We can see that this feedback-type ambiguity set \eqref{eq:feedback_B} is independent with $X_t = x$, as all information can be derived from the stock price $\vS_t$, and equivalently from $\vY_t$. Actually, the ambiguity set \eqref{eq:bayesian_ambiguity_set_general} can be reformulated as follows:
\begin{equation*}
\mc{B}:=
\left\{ \htB = (\htB_t)_{t \in [0,T]}:
\begin{array}{l}
\htB_t = \htB(t,\vY_t),\forall t\in [0,T], \\
\htB : [0,T]\times \mR^d \to \mR^d \text{ is Borel measurable}, \\
\htB(t,\vy) \in \mc{B}(t,\vy), \ \forall (t,\vy)\in [0,T]\times \mR^d
\end{array}
\right\}.
\end{equation*}
This leads to the following Hamilton–Jacobi–Bellman–Isaacs (HJBI) equation:

\begin{equation}\tag{HJBI-M}\label{eq:modified_HJBI_equation}
    \begin{cases}
        & \pt_t v(t,x,\vy) + \inf\limits_{\htB_t \in \mc{B}(t,\vy)} \sup\limits_{{\valpha}_t \in \mR^d} H^v(t,x,\vy,{\valpha}_t,\htB_t) = 0, \\ & \hfill  (t,x,\vy)\in [0,T)\times \mR \times \mR^d, \\
        & v(T,x,\vy)=U(x),   \hfill  (x,\vy)\in \mR \times \mR^d.
    \end{cases}
    \end{equation}

\begin{remark}
An interesting feature is that, although the infimum and supremum operators can be exchanged in Theorem~\ref{thm:mart_opt_principle}, such a max-min interchange does not generally remain valid in the HJBI equation. In particular, the classic max-min theorem fails for the HJBI formulation, and only the min-max structure yields a solvable equation suitable for further analysis. Nevertheless, despite this asymmetry, the resulting modified HJBI equation remains sufficient for establishing the verification theorem~\ref{thm:verification} and fully characterizing the value function and the optimal feedback strategy.
\end{remark}

\subsection{Semi-explicit Solution}\label{subsec:semi_explicit_solution_CARA}

In this subsection, we establish that the modified equation~\eqref{eq:modified_HJBI_equation} admits a classical solution with a semi-explicit representation, together with the form of the corresponding optimal feedback strategy. These results reveal several properties that are essential in proving the verification theorem in Section~\ref{sec:verification}, and connect the feedback optimizers of the modified HJBI equation with the optimizers of the original robust Bayesian optimization problem.

First, the following lemma allows us to derive an optimizer $\tilde{\valpha}$ for the inner supremum of \eqref{eq:modified_HJBI_equation}.

\begin{lemma}\label{lem:Hamilton_minmax_optimizer}
For all $(t,x,\vy) \in [0,T) \times \mR \times \mR^d$, we assume that $\partial_{x x}^2 v(t,x,\vy)<0$. Then for any $\htB_t \in \mc{B}(t,\vy)$, we have
\begin{equation*}
   H_1(t,x,\vy,\htB_t) := \sup_{{\valpha}_t \in \mR^d} H(t,x,\vy,{\valpha}_t,\htB_t) = H(t,x,\vy,\tilde{{\valpha}}_t,\htB_t) > -\infty 
\end{equation*}
with $\tilde{{\valpha}}_t = \tilde{{\valpha}}_t (t,x,\vy,\htB_t)
= - A(t,x,\vy)^{-1} \vK(t,x,\vy,\htB_t)$.


\end{lemma}
\begin{proof}
 As $\partial_{x x}^2 v(t,x,\vy)<0$, $A(t,x,\vy)$ is symmetric and negative definite. Completing the square yields
\begin{equation*}
\begin{aligned}
H(t,x,\vy,\valpha_t,\htB_t)
&=
\frac12
\big(\valpha_t + A(t,x,\vy)^{-1}\vK(t,x,\vy,\htB_t)\big)^\top
A(t,x,\vy)
\big(\valpha_t + A(t,x,\vy)^{-1}\vK(t,x,\vy,\htB_t)\big)
\\
&\quad
- \frac12\, \vK(t,x,\vy,\htB_t)^\top
A(t,x,\vy)^{-1}
\vK(t,x,\vy,\htB_t)
+ C(t,x,\vy,\htB_t).
\end{aligned}
\end{equation*}
For any fixed $(t,x,\vy,\htB_t)$, the Hamiltonian is strictly concave in ${\valpha}$, 
and the supremum over ${\valpha}$ can be computed explicitly. Consequently, the supremum over ${\valpha}_t$ is given by
\begin{equation*}
\begin{aligned}
H_1(t,x,\vy,\htB_t) =& \; \sup_{{\valpha}_t \in \mR^d} H(t,x,\vy,{\valpha}_t,\htB_t)\\
=&\;
-\frac12\, \vK(t,x,\vy,\htB_t)^\top
A(t,x,\vy)^{-1}
\vK(t,x,\vy,\htB_t) + C(t,x,\vy,\htB_t),
\end{aligned}
\end{equation*}
and the maximizer is explicitly given by
\begin{equation*}
\tilde{{\valpha}}_t = \tilde{{\valpha}}_t (t,x,\vy,\htB_t)
= - A(t,x,\vy)^{-1} \vK(t,x,\vy,\htB_t).
\end{equation*}
\end{proof}
\begin{remark}[Notational convention]
In this lemma, and in similar shorthand notation in this paper, the subscript $t$ in $\valpha_t$ and $\htB_t$ is used only to indicate that these quantities are stochastic processes evaluated at time $t$. The additional dependence on the time variable $t$ in expressions such as $\tilde{\valpha}_t(t,x,\vy,\htB_t)$ is kept for notational consistency and to make the feedback structure explicit. This convention is used only for stochastic processes and should not cause misunderstanding.
\end{remark}

Second, inspired by the close relation between the solution of the HJBI and the value function, we suppose that the function $v$ in the martingale optimality principle is of the form $v(t,x,\vy) = \Phi(t,\vy)U(x)$ and $ \Phi(t,\vy) > 0$. 
As the utility function $U$ is CARA, we obtain that $U \in C^2(\mR)$ is negative and strictly concave, and $(U'(x))^2 =  U(x)\,U''(x)$ for all $x \in \mR$.
As $ \Phi(t,\vy) > 0$, applying Lemma~\ref{lem:Hamilton_minmax_optimizer}, we have
\begin{equation*}
    \begin{aligned}
        H_1(t,x,\vy,\htB_t) :=& \sup_{{\valpha}_t \in \mR^d} H(t,x,\vy,{\valpha}_t,\htB_t) \\
         = & -\frac12\, \vK(t,x,\vy,\htB_t)^\top A(t,x,\vy)^{-1} \vK(t,x,\vy,\htB_t) + C(t,x,\vy,\htB_t) \\
= & -\frac12\, \lt(U'(x)\Phi(t,\vy)\htB_t +U'(x) \nabla_{\vy} \Phi(t,\vy)\Sigma \rt)^\top \lt(U''(x)\Phi(t,\vy)\Sigma \rt)^{-1} \\
 & \lt(U'(x)\Phi(t,\vy)\htB_t +U'(x) \nabla_{\vy} \Phi(t,\vy)\Sigma \rt) \\
&\; +\Big(U(x) \nabla_{\vy} \Phi(t,\vy)\Big)^\top \Big(\htB_t-\tfrac12(\vsigma \odot \vsigma)\Big) + \frac{1}{2} U(x) \operatorname{tr}\Big(\Sigma \nabla_{\vy}^2 \Phi(t,\vy)\Big) \\
= & \quad U(x)\Bigg[
-\frac{1}{2}\,\Phi(t,\vy)\,\htB_t^\top \Sigma^{-1}\htB_t
-\frac{1}{2 \Phi(t,\vy)}\,
\nabla_{\vy} \Phi(t,\vy)^\top\Sigma\nabla_{\vy} \Phi(t,\vy) \\
&\qquad
-\frac{1}{2}\,\nabla_{\vy} \Phi(t,\vy)^\top(\vsigma \odot \vsigma) 
+\frac{1}{2}\,\operatorname{tr}\!\big(\Sigma \nabla_{\vy}^2 \Phi(t,\vy)\big)
\Bigg].
    \end{aligned}
\end{equation*}
Then, by $U(x) \neq 0$, the HJBI equation~\eqref{eq:modified_HJBI_equation} reduces to the following HJB equation:
 \begin{equation}\label{eq:semi_explicit_HJB_equation}
    \begin{cases}
        & \pt_t \Phi(t,\vy) +\frac{1}{2}\,\operatorname{tr}(\Sigma \nabla_{\vy}^2 \Phi(t,\vy))
-\frac{1}{2 \Phi(t,\vy)}\,
\nabla_{\vy} \Phi(t,\vy)^\top\Sigma\nabla_{\vy} \Phi(t,\vy)
-\frac{1}{2}\,\nabla_{\vy} \Phi(t,\vy)^\top(\vsigma \odot \vsigma)
\\
& -  \inf\limits_{\htB_t \in \mc{B}(t,\vy)}   \Big\{ \frac{1}{2}\,\Phi(t,\vy)\,\htB_t^\top \Sigma^{-1}\htB_t \Big\}  = 0,  \, \qquad (t,\vy)\in [0,T) \times \mR^d, \\
        & \Phi(t,\vy) > 0, \qquad  (t,\vy)\in [0,T) \times \mR^d, \\
        & \Phi(T,\vy)=1,  \,  \qquad  \vy \in \mR^d.
    \end{cases}
\end{equation}

Finally, we present the existence theorem as follows, which establishes the existence of a classical solution to the reduced HJB equation, and consequently provides the existence of a classical solution to the modified HJBI equation~\eqref{eq:modified_HJBI_equation}.

\begin{theorem}\label{thm:existence_of_semisolution}
Suppose that the function
\begin{equation}\label{eq:gty_source_term}
        g(t,\vy) := \inf_{\htB_t \in \mc{B}(t,\vy)}   \Big\{ \frac{1}{2}\,\htB_t^\top \Sigma^{-1}\htB_t \Big\}
\end{equation}
belongs to $C_{loc}^{\alpha/2,\alpha}([0,T]\times \mathbb{R}^d)$ for some $\alpha \in (0,1)$, and that $g$ satisfies a polynomial growth condition: there exists some $m \in \mb{Z}_{+}$ such that
\begin{equation*}
    |g(t,\vy)| \le C(1+\nm{\vy}^m), \quad \forall (t,\vy)\in [0,T]\times \mathbb{R}^d.
\end{equation*}
Then, the HJB equation~\eqref{eq:semi_explicit_HJB_equation} admits at least one classical solution in $C^{1,2}([0,T]\times \mathbb{R}^d)$.
\end{theorem}
\begin{proof}
    To solve \eqref{eq:semi_explicit_HJB_equation}, we introduce the logarithmic transformation
\begin{equation*}
    F(t,\vy):=\log \Phi(t,\vy).
\end{equation*}
Then the HJB equation~\eqref{eq:semi_explicit_HJB_equation} reduces to the following equation:
 \begin{equation}\label{eq:reduced_semi_explicit_HJB_equation}
    \begin{cases}
        & \pt_t F(t,\vy) +\frac{1}{2}\,\operatorname{tr}(\Sigma \nabla_{\vy}^2 F(t,\vy))
-\frac{1}{2}\,\nabla_{\vy} F(t,\vy)^\top(\vsigma \odot \vsigma)
 - g(t,\vy)  = 0,  \,  (t,\vy)\in [0,T) \times \mR^d, \\        
        & F(T,\vy)=0,   \quad  \vy \in \mR^d,
    \end{cases}
\end{equation}
where the source term is 
\begin{equation*}
   g(t,\vy) := \inf_{\htB_t \in \mc{B}(t,\vy)}   \Big\{ \frac{1}{2}\,\htB_t^\top \Sigma^{-1}\htB_t \Big\}.
\end{equation*}
As Equation~\eqref{eq:reduced_semi_explicit_HJB_equation} is a semilinear parabolic equation with uniformly elliptic second-order coefficients, its solvability follows from classical second-order parabolic PDE theory (see, e.g., \cite{ladyzhenskaia1968linear}, \cite{krylov1987nonlinear}, \cite{lieberman1996second}). In particular, under the local H\"older continuity and polynomial growth assumptions imposed on the source term, one may apply the standard Schauder interior and global regularity estimates for second-order parabolic equations, together with approximation and continuation arguments, to establish the existence of a classical solution.
Thus, Equation \eqref{eq:reduced_semi_explicit_HJB_equation} admits a solution $F \in C_{loc}^{1+\alpha/2,2+\alpha}([0,T]\times \mathbb{R}^d) \subset C^{1,2}([0,T]\times \mathbb{R}^d)$.
Thus, the HJB equation~\eqref{eq:semi_explicit_HJB_equation} admits a classical solution $\Phi = e^{F} \in C^{1,2}([0,T]\times \mathbb{R}^d)$.
\end{proof}

\begin{remark}
    From the perspective of diffusion dynamics, the second-order term describes the spread of heat (or uncertainty), the first-order term captures the directed transport, while $g$ represents the local generation or absorption of energy within the system. Hence, $g$ is naturally referred to as a resource (or source) term, as it quantifies the instantaneous input accumulated along the system's evolution.
\end{remark}


Meanwhile, as $g(t,\vy)$ admits an optimizer whenever $\mc{B}(t,\vy)$ is convex and closed, if the HJB equation~\eqref{eq:semi_explicit_HJB_equation} admits a classical solution $\Phi(t,\vy) \in C^{1,2}([0,T]\times \mathbb{R}^d)$, we can denote the corresponding feedback optimizer by
\begin{equation}\label{eq:optimal_feedback_B}
   \htB_t^{*,v}= \htB_t^{*,v}(t,\vy):= \arg \min_{\htB_t \in \mc{B}(t,\vy)}   \Big\{ \frac{1}{2}\,\Phi(t,\vy)\,\htB_t^\top \Sigma^{-1}\htB_t \Big\} = \Phi(t,\vy)\, \cdot \arg \min_{\htB_t \in \mc{B}(t,\vy)}   \Big\{ \frac{1}{2}\,\htB_t^\top \Sigma^{-1}\htB_t \Big\}.
\end{equation}
And the optimal feedback control is given by
\begin{equation}\label{eq:optimal_feedback_alpha}
    \begin{aligned}
        {\valpha}^{*,v}_t(t,x,\vy):=\tilde{{\valpha}}_t^{v} (t,x,\vy,\htB_t^{*,v}(t,\vy)) & = - \frac{U(x)}{U'(x)} \Big[ \Sigma^{-1}\htB_t^{*,v}(t,\vy) + \frac{\nabla_{\vy} \Phi(t,\vy)}{\Phi(t,\vy)} \Big]\\
        & = \frac{1}{\gamma} \Big[ \Sigma^{-1}\htB_t^{*,v}(t,\vy) + \frac{\nabla_{\vy} \Phi(t,\vy)}{\Phi(t,\vy)} \Big].
    \end{aligned}
\end{equation}
\section{Verification Theorem and Optimal Control}\label{sec:verification}
In this section we establish the verification theorem for the robust Bayesian optimization problem. Meanwhile, building on the modified HJBI equation and its semi-explicit classical solution obtained in last section, we verify that the optimizer of \eqref{eq:modified_HJBI_equation} indeed coincides with that of the original robust Bayesian optimization problem.

Before the proof of the verification theorem, we first establish the following auxiliary result, which follows directly from the convexity and the quadratic structure of the ambiguity minimization~\eqref{eq:gty_source_term}.
\begin{lemma}\label{lem:diff_condition}
    For any $\htB_t \in \mc{B}(t,\vy)$, we have $
 \Phi(t,\vy)\,(\htB_t - \htB^{*,v}_t)^\top \Sigma^{-1} \htB^{*,v}_t   \ge 0.$
\end{lemma}
\begin{proof}
As $\mathcal{B}(t,\vy)$ is a convex set, and the objective function $ J(\vb) := \frac{1}{2}\vb^\top \Sigma^{-1}\vb $
is a strictly convex and differentiable function with respect to $\htB_t$ (due to the positive definiteness of $\Sigma^{-1}$), the optimal solution $\htB_t^{*,v}$ is uniquely determined by the first-order optimality condition for convex optimization. Specifically, for any $\htB_t \in \mathcal{B}(t,\vy)$, we  have $    \nabla_{\vb} J(\htB_t^{*,v})^\top (\htB_t - \htB_t^{*,v}) \ge 0.$
Computing the gradient and applying that $\Phi>0$, we obtain the required inequality.
\end{proof}

Then, we can establish the verification theorem as follows.

\begin{theorem}[Verification theorem]\label{thm:verification}
    Let $v(t,x,\vy) \in C^{1,2,2}([0,T] \times \mR \times \mR^d)$ satisfying \  $\partial_{x x}^2 v(t,x,\vy) < 0$ \ for all $(t,x,\vy) \in [0,T)\times \mR \times \mR^d$, and suppose that $v$ is a solution to \eqref{eq:modified_HJBI_equation}:
    \begin{equation*}
    \begin{cases}
        & \pt_t v(t,x,\vy) + \inf\limits_{\htB_t \in \mc{B}(t,\vy)} \sup\limits_{{\valpha}_t \in \mR^d} H^v(t,x,\vy,{\valpha}_t,\htB_t) = 0,  \quad  (t,x,\vy)\in [0,T)\times \mR \times \mR^d, \\
        & v(T,x,\vy)=U(x),   \quad  (x,\vy)\in \mR \times \mR^d.
    \end{cases}
    \end{equation*} 
    Then, when the feedback-type ambiguity set $\mc{B}(t,\vy)$ in \eqref{eq:feedback_B} is closed and convex, the corresponding feedback control ${\valpha}^{*,v}_t(t,x,\vy):=\tilde{{\valpha}}_t^{v} (t,x,\vy,\htB_t^{*,v}(t,\vy)): [0,T] \times \mR \times \mR^d \longrightarrow \mR^d$ given by \eqref{eq:optimal_feedback_alpha} and the feedback-type relaxed Bayesian estimator $\htB_t^{*,v}(t,\vy)$ given by \eqref{eq:optimal_feedback_B} form an optimal pair for the robust Bayesian problem \eqref{eq:robust_mpr_objective}.
\end{theorem}
\begin{proof}
Based on Theorem~\ref{thm:mart_opt_principle}, it suffices to check that the family of stochastic processes 
\begin{equation*}
    \{ V_t^{{\valpha},\htB} = v(t,X_t^{{\valpha},\htB},\vY_t) \, , 0 \leq t \leq T, \, {\valpha} \in \mathcal{A},\, \htB \in \mc{B}\},
\end{equation*}
where $v$ is a solution to \eqref{eq:modified_HJBI_equation}, satisfies the conditions of the martingale optimality principle with feedback optimizers ${\valpha}^{*,v}_t(t,x,\vy)$ and $\htB_t^{*,v}(t,\vy)$.

By the terminal equation in \eqref{eq:modified_HJBI_equation}, Condition (i) is already satisfied. By Itô's formula as in \eqref{eq:ito_results}, it suffices to check (ii) that under the feedback-type relaxed Bayesian estimator $\htB_t^{*,v}(t,\vy)$, for all ${\valpha} \in \mathcal{A}$, we have
$\E^{\mP^{\htB^*}} \Big[ \partial_t v(t,X_t^{\valpha},\vY_t) + H(t,X_t^{\valpha},\vY_t,{\valpha}_t,\htB_t^*) \Big] \le 0, \, \forall \, t\in [0,T) $, and (iii) that under the feedback control ${\valpha}^{*,v}_t(t,x,\vy)$, for all $\htB \in \mc{B}$, we have
$\E^{\mP^{\htB}} \Big[ \partial_t v(t,X_t^{{\valpha}^*},\vY_t) + H(t,X_t^{{\valpha}^*},\vY_t,{\valpha}_t^*,\htB_t) \Big] \ge 0, \, \forall \, t\in [0,T) $.

\vspace{4pt}
\noindent \textbf{Step1: Condition (ii) in Theorem~\ref{thm:mart_opt_principle}}.
As $v$ is a solution to \eqref{eq:modified_HJBI_equation}, and the feedback-type relaxed Bayesian estimator $\htB_t^*:=\htB_t^{*,v}(t,\vy)$ is a solution of the outside infimum for 
\begin{equation*}
    \inf\limits_{\htB_t \in \mc{B}(t,\vy)} \sup\limits_{{\valpha}_t \in \mR^d} H^v(t,x,\vy,{\valpha}_t,\htB_t),
\end{equation*}
we have
\begin{equation*}
    \pt_t v(t,x,\vy) + \sup_{{\valpha}_t \in \mR^d} H^v(t,x,\vy,{\valpha}_t,\htB_t^*) = 0,  \quad  \forall (t,x,\vy)\in [0,T)\times \mR \times \mR^d.
\end{equation*}
Consequently, for all ${\valpha}_t \in \mR^d$,
\begin{equation*}
    \pt_t v(t,x,\vy) + H^v(t,x,\vy,{\valpha}_t,\htB_t^*) \le 0,  \quad \forall (t,x,\vy)\in [0,T)\times \mR \times \mR^d.
\end{equation*}
Therefore, when $\htB^* := \{ \htB_t^{*,v}(t,\vY_t) \}$, for all ${\valpha} \in \mathcal{A}$, 
\begin{equation*}
    \E^{\mP^{\htB^*}} \lt[ \partial_t v(t,X_t^{\valpha},\vY_t) + H(t,X_t^{\valpha},\vY_t,{\valpha}_t,\htB_t^*) \rt] \le 0, \, \forall \, t\in [0,T].
\end{equation*}
This proves Condition (ii).

\vspace{4pt}
\noindent \textbf{Step2: Condition (iii) in Theorem~\ref{thm:mart_opt_principle}}. 
By Lemma~\ref{lem:Hamilton_minmax_optimizer}, $\tilde{{\valpha}}_t^v (t,x,\vy,\htB_t)$ is the optimizer of $\sup\limits_{{\valpha}_t \in \mR^d} H(t,x,\vy,{\valpha}_t,\htB_t)$  for any $\htB = (\htB_t)_{t \in [0,T]}\in \mc{B}$. 
Meanwhile, according to \eqref{eq:modified_HJBI_equation}, for ${\valpha}_t^* = {\valpha}^{*,v}_t(t,x,\vy)=\tilde{{\valpha}}_t^{v} (t,x,\vy,\htB_t^{*,v}(t,\vy))$, we have
\begin{equation*}
  \partial_t v(t,x,\vy) +H\lt(t,x,\vy,{\valpha}_t^*,\htB_t^{*,v}(t,\vy) \rt)  = 0,  \quad \forall (t,x,\vy)\in [0,T)\times \mR \times \mR^d.
\end{equation*}
Therefore,
\begin{equation*}
    \begin{aligned}
& \partial_t v(t,x,\vy) + H\lt(t,x,\vy,{\valpha}_t^*,\htB_t\rt) \\
= & \bigg[ \partial_t v(t,x,\vy) +H\lt(t,x,\vy,{\valpha}_t^*,\htB_t^{*,v}(t,\vy) \rt) \bigg] \\
& + \bigg[H\lt(t,x,\vy,\tilde{{\valpha}}_t^v (t,x,\vy,\htB_t),\htB_t \rt) -H\lt(t,x,\vy,{\valpha}_t^*,\htB_t^{*,v}(t,\vy) \rt)  \bigg] \\
&+\bigg[H\lt(t,x,\vy,{\valpha}_t^*,\htB_t\rt) - H\lt(t,x,\vy,\tilde{{\valpha}}_t^v (t,x,\vy,\htB_t),\htB_t \rt)\bigg]\\
= & \bigg[H_1\lt(t,x,\vy,\htB_t \rt) -H_1\lt(t,x,\vy,\htB_t^{*,v}(t,\vy) \rt)  \bigg] \\
& +\bigg[\frac{1}{2}\lt(\vK(t,x,\vy,\htB_t)-\vK(t,x,\vy,\htB_t^{*,v}(t,\vy))\rt)^\top A(t,x,\vy)^{-1} \lt(\vK(t,x,\vy,\htB_t)-\vK(t,x,\vy,\htB_t^{*,v}(t,\vy))\rt)\bigg]\\
= & \quad U(x)\bigg[-\frac{1}{2}\,\Phi(t,\vy)\,(\htB_t^\top \Sigma^{-1}\htB_t - \htB^{*,v \top}_t \Sigma^{-1}\htB^{*,v}_t) \bigg]
 +  \frac{ U(x) \Phi(t,\vy)}{2}\big[(\htB_t - \htB^{*,v}_t)^\top \Sigma^{-1}(\htB_t - \htB^{*,v}_t) \big]\\
= & - U(x) \big[ \Phi(t,\vy)\,(\htB_t - \htB^{*,v}_t)^\top \Sigma^{-1} \htB^{*,v}_t \big]\ge  0.
    \end{aligned}
\end{equation*}
The last inequality is due to $- U(x) \ge 0 $ and Lemma~\ref{lem:diff_condition}. Therefore, when ${\valpha}_t^* = {\valpha}^{*,v}_t(t,x,\vy)=\tilde{{\valpha}}_t^{v} (t,x,\vy,\htB_t^{*,v}(t,\vy))$, for all $\htB \in \mc{B}$, we have
\begin{equation*}
    \E^{\mP^{\htB}} \lt[ \partial_t v(t,X_t^{\valpha^*},\vY_t) + H(t,X_t^{\valpha^*},\vY_t,{\valpha}_t^*,\htB_t) \rt] \ge 0, \, \forall \, t\in [0,T],
\end{equation*}
i.e.,  Condition (iii) holds. Thus we complete the proof.
\end{proof}

Using the verification theorem, we obtain that the robust Bayesian optimization problem~\eqref{eq:robust_mpr_objective} admits an optimal feedback control (portfolio) \eqref{eq:optimal_feedback_alpha}:
\begin{equation*}
    \begin{aligned}
        {\valpha}^{*,v}_t(t,x,\vy) = \frac{1}{\gamma} \Bigg[ \Sigma^{-1}\htB_t^{*,v}(t,\vy) + \frac{\nabla_{\vy} \Phi(t,\vy)}{\Phi(t,\vy)} \Bigg],
    \end{aligned}
\end{equation*}
and a worst-case feedback-type relaxed Bayesian estimator \eqref{eq:optimal_feedback_B}:
\begin{equation*}
   \htB_t^{*,v}= \htB_t^{*,v}(t,\vy)= \Phi(t,\vy)\, \cdot \arg \min_{\htB_t \in \mc{B}(t,\vy)}   \Big\{ \frac{1}{2}\,\htB_t^\top \Sigma^{-1}\htB_t \Big\},
\end{equation*}
where $\Phi(t,\vy)$ is a solution for the HJB equation~\eqref{eq:semi_explicit_HJB_equation}.  Meanwhile, under such a feedback portfolio, we can solve the wealth dynamic $\{X_t\}_{t \in [0,T]}$ from
\eqref{eq:state_process_r_is_0}, which has a unique strong solution by \cite{karatzas1998brownian}. And the value function $V_0$ in \eqref{eq:robust_mpr_objective} can be derived from \eqref{eq:thm1-2} by the martingale optimality principle.

\section{Examples for the Semi-explicit Solution under Different Choices of the Ambiguity Set}\label{sec:example_and_extensions}

In this section, we verify that, under several representative choices of the ambiguity set, the induced feedback-type ambiguity set $\mc{B}(t,\vy)$ is convex and closed, and the corresponding source term function $g(t,\vy)$ satisfies the conditions required in Theorem~\ref{thm:existence_of_semisolution}. These properties are sufficient for ensuring the existence of classical solutions to the modified HJBI equation, the well-definedness of the associated feedback optimizers, and the validity of the verification theorem.

\subsection{Single-constraint Ambiguity: a Unified Feedback Structure}

For Examples~\ref{ex:L2_distribution}, \ref{ex:L2_expectation} and \ref{ex:L2_sample_path} in Section~\ref{subsec:ambiguity_setup}, we note a somewhat surprising fact: when the modified HJB equation~\eqref{eq:modified_HJBI_equation} is used to characterize the conditional problem at time $t$, these different ambiguity sets can all be unified to a same structure via the feedback-type ambiguity set
\begin{equation}\label{eq:feedback_L2_ambiguity}
\mc{B}(t,\vy)  :=
\left\{
\vb \in \mR^d : \mc{D}(\vb,\vvarphi(t,\vy)) \le \epsilon(t)
\right\}
\end{equation}
with a norm-based distance $\mc{D}(\vx,\vy) := \|\vx - \vy\|$. As all norms on $\mR^d$ are equivalent up to multiplicative scaling constants, without loss of generality, we assume that $\| \cdot \|$ is the 2-norm (Euclidean norm).

As the norm-based distance is continuous on $\mR^d$, $\mc{B}(t,\vy)$ is closed. Meanwhile, the convexity is obtained from triangle inequality and absolute homogeneity. Thus, $\mc{B}(t,\vy)$ is closed and convex under Examples~\ref{ex:L2_distribution}, \ref{ex:L2_expectation} and \ref{ex:L2_sample_path}. Then, in the following, we will prove that $g(t,\vy)$ satisfies the conditions in Theorem~\ref{thm:existence_of_semisolution}. To analyze perturbations of source term 
\begin{equation*}
    g(t,\vy) = \inf\limits_{\htB_t \in \mc{B}(t,\vy)}   \Big\{ \frac{1}{2}\,\htB_t^\top \Sigma^{-1}\htB_t \Big\}
\end{equation*}
with respect to $(t,\vy)$, we isolate its dependence
on the pointwise values of the functions $\vvarphi(\cdot,\cdot)$ and $\epsilon(\cdot)$.
For each fixed $(t,\vy)$, the quantities $(t,\vy)$ and $\epsilon(t)$ are fixed and treated as finite-dimensional parameters. With a slight abuse of
notation, we still denote these pointwise values by $\vvarphi$ and $\epsilon$,
respectively. This convention keeps the notation close to the original source term
and makes the correspondence transparent. We define
\begin{equation*}
G(\vvarphi,\epsilon)
:=
\inf_{\mc D(\vb,\vvarphi)\le \epsilon}
\left\{
\frac12\, \vb^\top \Sigma^{-1} \vb
\right\},
\qquad
(\vvarphi,\epsilon)\in \mR^d\times[0,\bar \epsilon].
\end{equation*}
Then, we have
\begin{equation*}
    g(t,\vy)
=
G\big(\vvarphi(t,\vy),\epsilon(t)\big)  \qquad \forall \, (t,\vy)\in[0,T]\times\mR^d.
\end{equation*}

In the following, we establish the stability of $G$ with respect to perturbations in the center
$\vvarphi$ and the tolerance $\epsilon$, which will then be used to prove the
regularity properties of the source term $g$.

\begin{proposition}[Perturbation with respect to the center $\vvarphi$]\label{prop:1_phi}
Suppose $\mc{K}$ is a compact subset of $\mR^d$. There exists a constant $K >0$ only depending on $\Sigma$ and $\mc{K}$ such that for any $\vvarphi_1,\vvarphi_2\in \mc{K}$ and any $\epsilon \in [0,\bar \epsilon]$,
\begin{equation*}
\begin{aligned}
|G(\vvarphi_1,\epsilon)-G(\vvarphi_2,\epsilon)|
\le\;&
K(\nm{\vvarphi_1}+\nm{\vvarphi_2}+\epsilon)
\nm{\vvarphi_1-\vvarphi_2}.
\end{aligned}
\end{equation*}
\end{proposition}
\begin{proof}

Let $\vdelta:=\vvarphi_2-\vvarphi_1$, and choose $
\vb_1
\in
\operatorname*{arg\,min}\limits_{\mc D(\vb,\vvarphi_1)\le\epsilon}
\frac12 \vb^\top\Sigma^{-1}\vb $. Define $\tilde \vb_1:=\vb_1+\vdelta$. Then,  using translation invariance of distance $\mc D$, we have $\mc D(\tilde \vb_1,\vvarphi_2)
=
\mc D(\vb_1,\vvarphi_1)
\le
\epsilon$.  Hence $\tilde \vb_1$ is feasible for $G(\vvarphi_2,\epsilon)$, and $
G(\vvarphi_2,\epsilon)\le
\frac12
\tilde \vb_1^\top\Sigma^{-1}\tilde \vb_1.
$ Then
\begin{equation*}
\begin{aligned}
    G(\vvarphi_2,\epsilon)-G(\vvarphi_1,\epsilon)
& \le
\frac12
(\tilde \vb_1^\top\Sigma^{-1}\tilde \vb_1
-
\vb_1^\top\Sigma^{-1}\vb_1)  =
\vb_1^\top\Sigma^{-1}\vdelta
+
\frac12
\vdelta^\top\Sigma^{-1}\vdelta  \le
K\nm{\vb_1}\nm{\vdelta}
+
K\nm{\vdelta}^2.
\end{aligned}
\end{equation*}
As $\mc D(\vb_1,\vvarphi_1)
\le
\epsilon$, we obtain  $\nm{\vb_1}
\le
\nm{\vvarphi_1}+\epsilon$. Furthermore,
\begin{equation*}
G(\vvarphi_2,\epsilon)-G(\vvarphi_1,\epsilon)
\le
K(\nm{\vvarphi_1}+\epsilon)\nm{\vvarphi_1-\vvarphi_2}
+
K\nm{\vvarphi_1-\vvarphi_2}^2.
\end{equation*}
Exchanging $\vvarphi_1,\vvarphi_2$ and applying that $\vvarphi_1,\vvarphi_2\in \mc{K}$, we obtain the required inequality.
\end{proof}

\begin{proposition}[Perturbation with respect to the tolerance $\epsilon$]\label{prop:2_epsilon}
There exists a constant $K >0$ only depending on $\Sigma$ such that for any $\vvarphi\in\mR^d$ and $ \epsilon_1,\epsilon_2 \in [0, \bar \epsilon]$,
\begin{equation*}
|G(\vvarphi,\epsilon_1)-G(\vvarphi,\epsilon_2)|
\le
K(\nm{\vvarphi}+\epsilon_1+\epsilon_2)
|\epsilon_1-\epsilon_2|.
\end{equation*}
\end{proposition}

\begin{proof}

Without loss of generality, we assume $\epsilon_1\le \epsilon_2$. Then
\begin{equation*}
\{\vb:\mc D(\vb,\vvarphi)\le\epsilon_1\}
\subset
\{\vb:\mc D(\vb,\vvarphi)\le\epsilon_2\},\ \mbox{and} \ G(\vvarphi,\epsilon_2)
\le
G(\vvarphi,\epsilon_1).
\end{equation*}
Let
\begin{equation*}
\vb_2
\in
\operatorname*{arg\,min}_{\mc D(\vb,\vvarphi)\le\epsilon_2}
\frac12 \vb^\top\Sigma^{-1}\vb.
\end{equation*}
If $\vb_2$ does not belong to the smaller ambiguity ball $\mc B_{\epsilon_1}(\vvarphi)
:=
\{\vb:\mc D(\vb,\vvarphi)\le \epsilon_1\}$, we define $\hat{\vb}_2$ as the radial projection of $\vb_2$ onto $\mc B_{\epsilon_1}(\vvarphi)$ along the line segment joining
$\vvarphi$ and $\vb_2$. Otherwise, we set $\hat{\vb}_2=\vb_2$. Then $\hat{\vb}_2$ is feasible under tolerance $\epsilon_1$, and
\begin{equation*}
\| \hat \vb_2-\vb_2 \|
\le
|\epsilon_2-\epsilon_1|.
\end{equation*}
Thus,
\begin{equation*}
\begin{aligned}
    0 \le G(\vvarphi,\epsilon_1)-G(\vvarphi,\epsilon_2)
& \le
\frac12
(\hat \vb_2^\top\Sigma^{-1}\hat \vb_2
-
\vb_2^\top\Sigma^{-1}\vb_2)\le
K\nm{\vb_2}\|\hat \vb_2-\vb_2\|
+
K\|\hat \vb_2-\vb_2\|^2.
\end{aligned}
\end{equation*}
As $\mc D(\vb_2,\vvarphi)
\le
\epsilon_2$, we obtain $\nm{\vb_2} \le \nm{\vvarphi}+\epsilon_2$. Then 
\begin{equation*}
G(\vvarphi,\epsilon_1)-G(\vvarphi,\epsilon_2)
\le
K(\nm{\vvarphi}+\epsilon_2)|\epsilon_1-\epsilon_2|
+
K|\epsilon_1-\epsilon_2|^2.
\end{equation*}
By $ \epsilon_1,\epsilon_2 \in [0, \bar \epsilon]$ is bounded, we obtain the required inequality.
\end{proof}

Then, we obtain the H\"older continuity and polynomial growth of function $g(t,\vy)$ in the following.

\begin{proposition}[Growth estimate and local H\"older continuity of $g(t,\vy)$]\label{prop:g_growth_holder}
Assume that there exists some $\alpha \in (0,1)$ such that
\begin{enumerate}
    \item $\epsilon(\cdot) \in C^{\alpha}[0,T]$ and is bounded on $[0,T]$, i.e., $0\le \epsilon(t) \le \bar\epsilon$;

    \item $\vvarphi(\cdot,\cdot)\in C^\alpha_{loc}([0,T]\times\mR^d)$, and is of polynomial growth.
\end{enumerate}
Then, we have the following growth estimate and local H\"older continuity of $g(\cdot,\cdot)$.

\begin{enumerate}
    \item[(i)] (Polynomial growth)
    There exist constants $K>0$ and $m>0$ such that
    \begin{equation*}
        |g(t,\vy)|
        \le
        K(1+\nm{\vy}^{\,m}),
        \qquad
       \forall \, (t,\vy)\in[0,T]\times\mR^d.
    \end{equation*}

    \item[(ii)] (Local H\"older continuity in $t$)
    For every compact set $\mc K\subset\mR^d$, there exists $K_{\mc K}>0$ such that
    \begin{equation*}
        |g(t_1,\vy)-g(t_2,\vy)|
        \le
        K_{\mc K}|t_1-t_2|^{\alpha},
        \qquad
       \forall \, \vy\in\mc K, \ t_1,t_2 \in [0,T].
    \end{equation*}

    \item[(iii)] (Local H\"older continuity in $\vy$)
    For every compact set $\mc K\subset\mR^d$, there exists $K_{\mc K}>0$ such that
    \begin{equation*}
        |g(t,\vy_1)-g(t,\vy_2)|
        \le
        K_{\mc K}\nm{\vy_1-\vy_2}^{\alpha},
        \qquad
      \forall \,  \vy_1,\vy_2\in\mc K, \ t \in [0,T].
    \end{equation*}
\end{enumerate}
Hence $g\in C_{\mathrm{loc}}^{\alpha,\alpha}([0,T]\times\mR^d) \subset C_{\mathrm{loc}}^{\alpha/2,\alpha}([0,T]\times\mR^d)$ is locally H\"older continuous, and of polynomial growth.
\end{proposition}

\begin{proof}
\textbf{Step 1: Growth estimate.}
Note that
\begin{equation*}
g(t,\vy)
=
G(\vvarphi(t,\vy),\epsilon(t)).
\end{equation*}

\noindent Using $|G(\vvarphi,\epsilon)-G(0,\epsilon)|
\le
K(\nm{\vvarphi}+\epsilon)\nm{\vvarphi}$, we obtain
\begin{equation*}
|G(\vvarphi,\epsilon)|
\le
|G(0,\epsilon)|
+
K(\nm{\vvarphi}+\epsilon)\nm{\vvarphi}.
\end{equation*}

\noindent As $\epsilon(t)\le\bar\epsilon$ and $G(0,\epsilon)=0$, it follows that
\begin{equation}\label{eq:g_growth_1}
|g(t,\vy)|
\le
K(1+\nm{\vvarphi(t,\vy)}^2).
\end{equation}
Meanwhile, as $\vvarphi$ is of polynomial growth, by \eqref{eq:g_growth_1}, $g$ has polynomial growth.\\
\textbf{Step 2: Perturbation with respect to $t$.}
\noindent For any $t_1,t_2 \in [0,T]$ and any fixed $\vy$, 
\begin{equation*}
|g(t_1,\vy)-g(t_2,\vy)|
=
|G(\vvarphi(t_1,\vy),\epsilon(t_1))
-
G(\vvarphi(t_2,\vy),\epsilon(t_2))|.
\end{equation*}

\noindent Using the above perturbation estimate of function $G$, we have
\begin{equation*}
\begin{aligned}
|g(t_1,\vy)-g(t_2,\vy)|
\le\;&
K(\nm{\vvarphi(t_1,\vy)}+\nm{\vvarphi(t_2,\vy)}+\bar\epsilon)
\nm{\vvarphi(t_1,\vy)-\vvarphi(t_2,\vy)}
\\
&+
K(\nm{\vvarphi(t_1,\vy)}+\epsilon(t_1)+\epsilon(t_2))
|\epsilon(t_1)-\epsilon(t_2)|.
\end{aligned}
\end{equation*}

\noindent As $\vvarphi(t,\vy)\in C^\alpha_{loc}([0,T]\times\mR^d)$, for any compact set $\mc K$, we have
\begin{equation*}
\nm{\vvarphi(t_1,\vy)-\vvarphi(t_2,\vy)}
\le
K_{\mc K}|t_1-t_2|, \qquad \forall \, \vy \in \mc K, \  t_1,t_2 \in [0,T].
\end{equation*}
Meanwhile, as $\epsilon(\cdot)$ is H\"older continuous,  i.e., $
|\epsilon(t_1)-\epsilon(t_2)|
\le
K|t_1-t_2|^\alpha, \,\; \forall \, t_1,t_2 \in [0,T]$, we have
\begin{equation*}
|g(t_1,\vy)-g(t_2,\vy)|
\le
K_{\mc K}|t_1-t_2|^\alpha.
\end{equation*}
\textbf{Step 3: Perturbation with respect to $\vy$.} \ 
Similar to Step 2, for any $\vy_1,\vy_2 \in [0,T]$ and any fixed $t$, we have
\begin{equation*}
\begin{aligned}
|g(t,\vy_1)-g(t,\vy_2)|
& =
|G(\vvarphi(t,\vy_1),\epsilon(t))
-
G(\vvarphi(t,\vy_2),\epsilon(t))|\\
& \le 
K(\nm{\vvarphi(t,\vy_1)}+\nm{\vvarphi(t,\vy_2)}+\epsilon(t))
\nm{\vvarphi(t,\vy_1)-\vvarphi(t,\vy_2)}.
\end{aligned}
\end{equation*}

\noindent As $\vvarphi(t,\vy)\in C^\alpha_{loc}([0,T]\times\mR^d)$, for any  compact $\mc K$, we have
\begin{equation*}
\nm{\vvarphi(t,\vy_1)-\vvarphi(t,\vy_2)}
\le
K_{\mc K}\nm{\vy_1-\vy_2}.
\end{equation*}
Then,
\begin{equation*}
|g(t,\vy_1)-g(t,\vy_2)|
\le
K_{\mc K}\nm{\vy_1-\vy_2}.
\end{equation*}
Thus $g$ is locally H\"older continuous in $\vy$.
Combining all estimates proves the proof. \end{proof}

Then, we verify that the conditions in Proposition~\ref{prop:g_growth_holder} hold. Under the sub-Gaussian condition~\eqref{eq:sub_gaussian_condition}, we have that $\vvarphi(\cdot, \cdot) \in C^\infty([0,T]\times\mR^d)$ and is of polynomial growth. And the tolerance level function $\epsilon(\cdot)$ is assumed to be $\alpha$-H\"older continuous and decreasing, and thus bounded. Therefore, by Proposition~\ref{prop:g_growth_holder}, the corresponding function $g(t,\vy)$ satisfies the conditions required in Theorem~\ref{thm:existence_of_semisolution}.


\subsection{Multiple-constraint Ambiguity: a Natural Extension}

This subsection considers the multiple-constraint ambiguity setting, which is a direct extension of the single-constraint framework studied above. In this case, the ambiguity set is obtained as the intersection of several single-constraint ambiguity sets, while the overall analytical structure remains essentially unchanged. The corresponding perturbation results, regularity properties, and existence arguments are therefore analogous to those in the single-constraint case, up to minor technical modifications. Consequently, the proofs follow largely the same line of reasoning, showing that the robust Bayesian framework can naturally extend to more general ambiguity specifications.

When the multiple-constraint ambiguity is taken as in Example~\ref{ex:multi_constraints}, it gives rise to the following feedback-type ambiguity set:
\begin{equation}\label{eq:feedback_L2_ambiguity_multi}
\mc{B}(t,\vy) :=
\left\{
\vb \in \mR^d : \mc{D}_i(\vb,\vvarphi(t,\vy)) \le \epsilon_{i}(t),
\, i=1,2,\cdots,n
\right\},
\end{equation}
where, in this example, each $\mc{D}_i(\vx,\vy):=\|\vx-\vy\|$ is taken to be the Euclidean norm on $\mR^d$. More generally, the subsequent arguments remain valid when the distances $\mc{D}_i$ are induced by different norms, since this only leads to finitely many multiplicative rescalings of the corresponding constraints and does not affect the structure of the proof. We also assume that $\vepsilon(t):=(\epsilon_1(t),\epsilon_2(t),\cdots,\epsilon_n(t))$ is decreasing and non-negative in each component (thus bounded), and is $\alpha$-H\"older continuous for some $\alpha\in(0,1)$.

The feedback-type ambiguity set in~\eqref{eq:feedback_L2_ambiguity_multi} can be viewed as the intersection of multiple single-constraint ambiguity sets, and is therefore closed and convex. Hence, under this multiple-constraint feedback-type ambiguity set, the preceding analysis of the martingale optimality principle, the HJBI equation, and the verification theorem remain valid because those arguments only use the fact that $\mc{B}(t,\vy)$ is closed and convex.

It suffices to verify the stability properties of the corresponding source function
\begin{equation*}
    g(t,\vy)
:=
\inf_{\htB_t\in\mc{B}(t,\vy)}
\left\{
\frac12\,\htB_t^\top\Sigma^{-1}\htB_t
\right\}.
\end{equation*}
Define
\begin{equation*}
G(\vvarphi,\vepsilon)
:=
\inf_{\substack{\mc D_i(\vb,\vvarphi)\le \epsilon_i, \\ i=1,2,\cdots,n}}
\left\{
\frac12\, \vb^\top \Sigma^{-1} \vb
\right\},
\qquad
(\vvarphi,\vepsilon)\in \mR^d\times[0, \bar \epsilon]^n .
\end{equation*}
Then $G$ satisfies stability properties analogous to those in the single-constraint case, as stated below.

\begin{proposition}[Perturbation with respect to the center $\vvarphi$]\label{prop:1_phi_m}
There exists a constant $K >0$ only depending on $\Sigma$ such that for any $\vvarphi_1,\vvarphi_2\in\mR^d$ and any fixed $\vepsilon \in [0,\bar \epsilon]^n$,
\begin{equation*}
\begin{aligned}
|G(\vvarphi_1,\vepsilon)-G(\vvarphi_2,\vepsilon)|
\le\;&
K(\nm{\vvarphi_1}+\nm{\vvarphi_2}+ \nm{\vepsilon} )
\nm{\vvarphi_1-\vvarphi_2}.
\end{aligned}
\end{equation*}
\end{proposition}
\begin{proof}
The proof follows the same argument as that of
Proposition~\ref{prop:1_phi}. The only modification is that the multiple constraints 
\begin{equation*}
    \mc D_i(\vb_1,\vvarphi_1)
\le
\epsilon_i, \, i=1,2,\cdots,n,
\end{equation*}
yield the estimate $\nm{\vb_1}
\le
\nm{\vvarphi_1}+ \nm{\vepsilon}$. With this estimate, the rest of the proof is unchanged.
\end{proof}

\begin{proposition}[Perturbation with respect to the tolerance $\epsilon$]\label{prop:2_epsilon_m}
There exists a constant $K >0$ only depending on $\Sigma$ such that for any fixed $\vvarphi\in\mR^d$ and $\vepsilon_1, \vepsilon_2 \in [0, \bar\epsilon ]^n $ satisfying $\epsilon_{1,i} \le \epsilon_{2,i}$, $i=1,2,\cdots,n$, we have
\begin{equation*}
|G(\vvarphi,\vepsilon_1)-G(\vvarphi,\vepsilon_2)|
\le
K(\nm{\vvarphi}+\nm{\vepsilon_1} + \nm{\vepsilon_2})
\nm{\vepsilon_1-\vepsilon_2}.
\end{equation*}
\end{proposition}

\begin{proof}
The proof procedure is almost identical to that of Proposition~\ref{prop:2_epsilon} and we only indicate the needed modifications. For $k=1,2$, write
$\vepsilon_k=(\epsilon_{k,1},\ldots,\epsilon_{k,n})$ and let
\begin{equation*}
\mc B_{\vepsilon_k}(\vvarphi)
:=
\{\vb:\mc D_i(\vb,\vvarphi)\le \epsilon_{k,i},\ i=1,\cdots,n\}.
\end{equation*}
Then we have 
\begin{equation*}
\mc B_{\vepsilon_1}(\vvarphi)
\subset
\mc B_{\vepsilon_2}(\vvarphi),\ \mbox{and} \ G(\vvarphi,\vepsilon_2)
\le
G(\vvarphi,\vepsilon_1).
\end{equation*}
Let
\begin{equation*}
\vb_2
\in
\operatorname*{arg\,min}_{\vb\in \mc B_{\vepsilon_2}(\vvarphi)}
\frac{1}{2}\vb^\top\Sigma^{-1}\vb .
\end{equation*}
The feasibility of $\vb_2$ under $\vepsilon_2$ gives
\begin{equation*}
\|\vb_2\|
\le
\|\vvarphi\|+\|\vepsilon_2\|.
\end{equation*}
Choose
\begin{equation*}
j
\in
\operatorname*{arg\,min}_{1\le i\le n}\epsilon_{1,i}.
\end{equation*}
We define $\hat{\vb}_2$ by projecting $\vb_2$ radially toward $\vvarphi$ onto
the ball with radius $\epsilon_{1,j}$. More precisely, if
$\vb_2=\vvarphi$, set $\hat{\vb}_2=\vb_2$; otherwise, set
\begin{equation*}
\hat{\vb}_2
=
\vvarphi
+
\lambda(\vb_2-\vvarphi),
\qquad
\lambda
:=
\min\left\{
1,
\frac{\epsilon_{1,j}}{\|\vb_2-\vvarphi\|}
\right\}.
\end{equation*}
Then $\hat{\vb}_2\in \mc B_{\vepsilon_1}(\vvarphi)$. Moreover, as
$\vb_2\in \mc B_{\vepsilon_2}(\vvarphi)$, we have
\begin{equation*}
\|\hat{\vb}_2-\vb_2\|
=
\bigl(\|\vb_2-\vvarphi\|-\epsilon_{1,j}\bigr)^+
\le
\bigl(\epsilon_{2,j}-\epsilon_{1,j}\bigr)^+
\le
|\epsilon_{2,j}-\epsilon_{1,j}|
\le
K\|\vepsilon_2-\vepsilon_1\|.
\end{equation*}
Using the same quadratic estimate as in the proof of
Proposition~\ref{prop:2_epsilon}, we obtain
\begin{equation*}
G(\vvarphi,\vepsilon_1)-G(\vvarphi,\vepsilon_2)
\le
K(\|\vvarphi\|+\|\vepsilon_2\|)
\|\vepsilon_1-\vepsilon_2\|
+
K\|\vepsilon_1-\vepsilon_2\|^2 .
\end{equation*}
As $\vepsilon_1, \vepsilon_2 \in [0, \bar\epsilon ]^n $ are bounded, we obtain the desired estimate.
\end{proof}

Finally, similar to Proposition~\ref{prop:g_growth_holder}, when $\vepsilon \in C^{\alpha}[0,T]$ is decreasing in every component, we can obtain that the corresponding $g(t,\vy) := \inf\limits_{\htB_t \in \mc{B}(t,\vy)}   \Big\{ \frac{1}{2}\,\htB_t^\top \Sigma^{-1}\htB_t \Big\}$ satisfies all conditions in Theorem~\ref{thm:existence_of_semisolution} as well.

\section{Conclusion}\label{sec:conclusion}

This paper studies a continuous-time robust Bayesian optimization problem under drift uncertainty, combining Bayesian learning with discrepancy-based robustness to address uncertainty around posterior estimates. Introducing ambiguity sets directly around Bayesian estimators and reformulating the problem through a feedback-type ambiguity framework, we overcome the time inconsistency inherent in the dynamic robust Bayesian optimization problem. This leads to a modified HJBI equation, from which we characterize the value function and the optimal feedback strategy through PDE methods.


Our framework accommodates a broad class of ambiguity structures, including Wasserstein, $L^p$, pathwise, and multiple-constraint formulations, within a unified analytical setting. Although we focus mainly on norm-based distances, the analysis only relies on certain structural properties of the induced feedback ambiguity set, namely convexity, closedness, and the translation-invariance arguments used to establish the required regularity of $g(t,\vy)$. Hence, other distance specifications may also be incorporated into the framework, provided that these properties are preserved. These results provide a general methodology for integrating Bayesian learning and robust control in continuous time, and offer a foundation for future extensions to more general market models and dynamic decision-making problems with partial information.

This Bayesian portfolio optimization model can naturally extend to other partial information frameworks where robustness can protect the decision maker from ambiguity, as we have only used the fact that the estimator can be represented via observable information. Meanwhile, if we consider the framework where the drift and volatility are both unknown, the procedure are also similar, as the partial information structure and time-varying feedback-type ambiguity are still applicable. The only difference is that we need to consider a more complicate linear-quadratic optimization in the source term. Future research may extend this framework to settings in which the HJBI equation no longer admits classical solutions. A typical example is CRRA utility, where one expects viscosity solutions rather than classical ones, which is an open problem in time-inconsistent frameworks.



\subsubsection*{Acknowledgement} 
Zongxia Liang is supported by the National Natural Science Foundation of China under grant no. 12271290. Yang Liu acknowledges financial support from the National Natural Science Foundation of China (Grant No. 12401624), Guangdong Science and Technology Program (Grant No. 2024QN11X076), Shenzhen Science and Technology Program (Grant No. RCBS20231211090814028, JCYJ20250604141203005, 2025TC0010) as well as The Chinese University of Hong Kong (Shenzhen) University Development Fund (Grant No. UDF01003336) and is partly supported by the Guangdong Provincial Key Laboratory of Mathematical Foundations for Artificial Intelligence (Grant No. 2023B1212010001). 

\bibliography{refs}
\bibliographystyle{plainnat} 

\appendix

\end{document}